\documentclass[12pt]{amsart}

\usepackage{amsmath,amsthm,amssymb,color}
\usepackage{pdfsync}
\usepackage[latin1]{inputenc}
\usepackage{graphicx}
\usepackage{pstricks,epic,eepic}
\usepackage{esint}
\usepackage{xcolor}
\usepackage{enumitem}
\usepackage{geometry,}
\usepackage{marginnote}
\usepackage{mathrsfs}
\usepackage{pgfplots}
\usepackage[colorlinks]{hyperref}
\hypersetup{colorlinks,
linkcolor={red!50!black},
citecolor={blue!50!black},
urlcolor={blue!80!black}
}

\usetikzlibrary{arrows}
\newtheorem{theorem}{Theorem}[section]
\newtheorem*{theorem*}{Theorem}
\newtheorem*{lemma*}{Lemma}
\newtheorem*{theorema*}{Theorem A}
\newtheorem*{theoremb*}{Theorem B}
\newtheorem*{theoremc*}{Theorem C}
\newtheorem*{mainlemma*}{Main Lemma}

\newtheorem{lemma}[theorem]{Lemma}

\newtheorem{corollary}[theorem]{Corollary}
\newtheorem{proposition}[theorem]{Proposition}

\newtheorem*{claim*}{Claim}
\newtheorem*{conjecture*}{Conjecture}
\theoremstyle{definition}
\newtheorem{definition}[theorem]{\bf Definition}

\theoremstyle{remark}
\newtheorem{remark}[theorem]{\bf Remark}

\numberwithin{equation}{section}

\newcommand{\vv}{\vspace{2mm}}
\newcommand{\vvv}{\vspace{4mm}}

\newcommand{\weakL}{L^{1,\infty}}
\newcommand{\Dsigma}{\mathcal{D}_\sigma}

\newcommand{\supp}{\operatorname{supp}}
\newcommand{\dist}{{\operatorname{dist}}}
\newcommand{\pom}{{\partial\Omega}}
\newcommand{\diam}{\operatorname{diam}}
\newcommand{\R}{\mathbb R}
\newcommand{\HH}{\mathcal H}
\newcommand{\WW}{\mathcal W}
\newcommand{\wt}{\widetilde}
\newcommand{\haj}{Haj{\l}asz}

\newcommand{\Top}{\operatorname{Top}}

\newcommand{\Lip}{\operatorname{Lip}}
\begin{document}

\title
[Rellich inequalities, Regularity and uniform rectifiability]
{One-sided Rellich inequalities, Regularity problem and uniform rectifiability}

\author[Josep M. Gallegos]{Josep M. Gallegos}
\address{Departament de Matem\`atiques\\ Universitat Aut\`onoma de Barcelona
\\ Edifici C Facultat de Ci\`encies\\08193 Bellaterra
}
\email{jgallegosmath@gmail.com}
\thanks{J.G. was supported by the European Research Council (ERC) under the European Union's Horizon 2020 research and innovation programme (grant agreement 101018680) and partially supported by MICINN (Spain) under grant PID2020-114167GB-I00, and 2021-SGR-00071 (Catalonia).}


\begin{abstract}
Let $\Omega\subset \R^{n+1}$, $n\geq1$,
be a bounded open set satisfying the interior corkscrew condition with a {uniformly $n$-rectifiable boundary} but without any connectivity assumptions. We establish the estimate
\[
\Vert \partial_\nu u_f \Vert_{\mathcal M} \lesssim \Vert \nabla_H f \Vert_{L^1(\pom)}, \quad \mbox{for all $f\in\Lip(\pom)$}
\]
where $u_f$ is the solution to the Dirichlet problem with boundary data $f$, $\partial_\nu u_f$ is the normal derivative of $u_f$ at the boundary in the weak sense, $\Vert \cdot \Vert_{\mathcal M}$ denotes the total variation norm and $\nabla_H f$ is the \haj-Sobolev gradient of $f$.

Conversely, {if $\Omega\subset \R^{n+1}$ is a corkscrew domain with $n$-Ahlfors regular boundary and the previous inequality holds for solutions to the Dirichlet problem on $\Omega$, then $\pom$ must satisfy the weak-no-boxes condition} introduced by David and Semmes. 
Hence, in {the planar case}, the one-sided Rellich inequality characterizes the uniform rectifiability of $\pom$.

We also show solvability of the regularity problem in weak $L^1$ for bounded corkscrew domains with a uniformly $n$-rectifiable boundary, that is
\[\Vert N(\nabla u_f) \Vert_{L^{1,\infty}(\pom)} \lesssim \Vert \nabla_H f\Vert_{L^1(\pom)},\quad \mbox{for all $f\in\Lip(\pom)$}\]
where $N$ is the nontangential maximal operator.
As an application of our results, we prove that for general elliptic operators, the solvability of the Dirichlet problem does not imply the solvability of the regularity problem.
\end{abstract}

\maketitle
%
%
%
%
%
%
%
%
%



\section{Introduction}

In this paper, we study an $L^1$ one-sided Rellich inequality in bounded corkscrew domains $\Omega\subset \R^{n+1}$ with uniformly $n$-rectifiable boundary. 
Rellich inequalities, which relate the normal derivative and the tangential derivative of a harmonic function, have been widely used to study the Dirichlet, regularity and Neumann problems in Lipschitz domains (pioneered by D. Jerison and C.Kenig in \cite{JK}) and, more recently, in domains with uniformly $n$-rectifiable boundary (see the work of M. Mourgoglou and X. Tolsa \cite{MT} where an $L^p$ version is proved for $p>1$). 
In the present work, we investigate an endpoint version of the one-sided Rellich inequalities:
\[
\Vert \partial_\nu u_f \Vert_{L^p(\pom)} \lesssim \Vert \nabla_H f \Vert_{L^p(\pom)}, \quad \mbox{for all $f\in\Lip(\pom)$},
\] 
proved in \cite{MT}. {Here $u_f$ denotes the solution to the Dirichlet problem for the Laplacian on $\Omega$, $\partial_\nu u_f$ is its normal derivative understood in a weak sense and $\nabla_H f$ is the \haj{} gradient of the boundary data $f$.}
However the {version} we consider is qualitatively different from the ones in 
\cite{MT}. Their inequalities are closely tied to the weak-$\mathcal A_\infty$ properties of the harmonic measure of the domain whereas ours is not. In particular, our inequality holds without any weak connectivity assumptions on the domain, in contrast to the $L^p$ case that requires the \textit{weak local John} {condition} (the full geometric characterization of the weak-$\mathcal A_\infty$ properties of the harmonic measure was proved by J. Azzam, S. Hofmann, J.M. Martell, M. Mourgoglou and X. Tolsa in \cite{AHMMT}).
On the other hand, we show that the validity of our inequality for all Lispchitz functions on the boundary of the domain implies a certain geometric property of the boundary called the \textit{weak-no-boxes property}.
In the special case of dimension $n+1=2$, David and Semmes \cite{DS2} proved that this weak-no-boxes property is equivalent to uniform $1$-rectifiability of the boundary. Hence, in this case, the inequality characterizes uniform rectifiability. The $L^p$ case of the inequality also implies uniform rectifiability of the boundary, {although it is not equivalent to it (see the results of S. Hofmann, P. Le,  J.M. Martell and K. Nystr\"om \cite{HLMN} and M. Mourgoglou and X. Tolsa \cite{MT2})}. 
In our setting, the failure of the inequality in non-uniformly rectifiable domains is due to the presence of Lipschitz functions on the boundary that behave badly at too many scales, a phenomenon related to the so-called WALA property introduced by G. David and S. Semmes in \cite[Definition 2.47]{DS2}. At least, on the surface this mechanism appears to be quite different {to the ones used to prove that the $L^p$ versions imply uniform rectifiability}.

The Rellich inequality we consider can also be understood as an endpoint version of the regularity problem. Closely related to this, we establish a weak $L^1$ estimate for the regularity problem of the form 
\[
\Vert N(\nabla u_f ) \Vert_{L^{1,\infty}(\pom)} \lesssim \Vert \nabla_H f \Vert_{L^1(\pom)}, \quad \mbox{for all $f\in\Lip(\pom)$},
\] 
which again only requires uniform $n$-rectifiability of the boundary but no connectivity assumptions on the domain. This contrasts with the $L^1$ endpoint version studied by M. Mourgoglou, X.Tolsa and the author in \cite{GMT} {which should be understood as a Hardy space $H^1-L^1$ result.} Our $H^1-L^{1,\infty}$ estimate serves as a suitable starting point for extrapolating the solvability of the regularity problem and hence implies solvability in $L^p$ for $1-\epsilon<p<1$. Conversely, we show that the solvability of the regularity problem in weak $L^1$ is nontrivial, as it fails in the complement of the $4$-corners Cantor set. 
This result is not only interesting in its own right, but it also allows us to prove that the Dirichlet problem and the regularity problem are not equivalent in general domains. This follows by invoking the specific construction of G. David and S. Mayboroda in \cite{DM} of an elliptic operator in the $4$-corners Cantor set such that its elliptic measure is comparable to the surface measure. {The author thinks these results may provide further insight on the relationship between the regularity problem and the Dirichlet problem for general operators.}

\vv

\subsection{Definitions}\label{section:definitions}
We introduce some definitions and notations.
Let $X$ be a measure space equipped with a measure $\mu$. The space weak $L^1$ is the space of measurable functions satisfying
\[
\sup_{\lambda>0} \lambda \,\sigma(\{x \in X : |f(x)|>\lambda\}) < +\infty
\]
and we denote it by $L^{1,\infty}(X)$.
The $L^{1,\infty}$ quasi-norm of a function $f$ is
\[
\Vert f \Vert_{L^{1,\infty}} = \sup_{\lambda>0} \lambda \sigma(\{x \in X : |f(x)|>\lambda\})
\]
(note that it is not a norm since the triangle inequality does not hold).
 
A set $E\subset \R^{n+1}$ is called $n$-{\textit {rectifiable}} if there are Lipschitz maps
$f_i:\R^n\to\R^{n+1}$, $i=1,2,\ldots$, such that 
$$
\HH^n\biggl(E\setminus\bigcup_i f_i(\R^n)\biggr) = 0,
$$
where $\HH^n$ stands for the $n$-dimensional Hausdorff measure. We will assume $\HH^n$ to be normalized so that it coincides with $n$-dimensional Lebesgue measure in
$\R^n$.

All measures in this paper are assumed to be Radon measures.
A measure $\mu$ in $\R^{n+1}$ is called 
$n$-{\textit {Ahlfors}} {\textit {regular}}  if there exists some
constant $C_{0}>0$ such that
$$C_0^{-1}r^n\leq \mu(B(x,r))\leq C_0\,r^n\quad \mbox{ for all $x\in
	\supp\mu$ and $0<r\leq \diam(\supp\mu)$.}$$

The measure $\mu$ is \textit{uniformly  $n$-rectifiable} if it is 
$n$-Ahlfors regular and
there exist constants $\theta, M >0$ such that for all $x \in \supp\mu$ and all $0<r\leq \diam(\supp\mu)$ 
there is a Lipschitz mapping $g$ from the ball $B_n(0,r)$ in $\R^{n}$ to $\R^{n+1}$ with $|g|_{\operatorname{Lip}} \leq M$ such that
\[
\mu (B(x,r)\cap g(B_{n}(0,r)))\geq \theta r^{n}
\]
where $|g|_{\operatorname{Lip}}$ is the Lipschitz seminorm of $g$.

A set $E\subset \R^{n+1}$ is $n$-Ahlfors regular if $\HH^n|_E$ is $n$-Ahlfors regular. 
{A set $G \subset E \times \mathbb R_+$ with $E$ $n$-Ahlfors regular is Carleson if there exists $C>0$ such that for all $x\in E$ and $0<t<\diam E$, we have $\int_{E\cap B(x,t)} \int_0^{t} \chi_{G}(x,r) \, \frac{dr}{r} \, d\HH^n|_E \leq Ct^n$. }

Also, 
$E$ is uniformly $n$-rectifiable if $\HH^n|_E$ is uniformly $n$-rectifiable.

The notion of uniform rectifiability should be considered a quantitative version of rectifiability. It was introduced in the pioneering works \cite{DS1} and \cite{DS2} of David and Semmes, who were seeking a good geometric framework under which all singular integrals with odd and sufficiently smooth kernels are bounded in $L^2$.

\vv

Following \cite{JK2}, we say that $\Omega \subset \R^{n+1}$ satisfies the \textit{corkscrew condition}, or that it is a corkscrew open set or domain if
there exists some $c>0$ such that for all
$x\in\pom$ and all $r\in(0, 2\diam(\Omega))$ there exists a ball $B\subset B(x,r)\cap\Omega$ such that
$r(B)\geq c\,r$. We say that $\Omega$ is \textit{two-sided corkscrew} if both $\Omega$ and $\R^{n+1}\setminus
\overline\Omega$ satisfy the corkscrew condition.

If $X$ is a metric space, given an interval $I\subset\R$, any continuous $\gamma:I \to X$ is called {\it path}. A path of finite length is called {\it rectifiable path}.

Following \cite{HMT}, we say that the open set $\Omega  \subset \R^{n+1}$ satisfies the {\it  local John condition} if there is $\theta\in(0,1)$
such that the following holds: For all $x\in\pom$ and $r\in(0,2\diam(\Omega))$ 
there is $y\in B(x,r)\cap \Omega$ such that $B(y,\theta r)\subset \Omega$ with the property that for all
$z\in B(x,r)\cap\pom$ one can find a rectifiable {path} $\gamma_z:[0,1]\to\overline \Omega$ with length
at most $\theta^{-1}|x-y|$ such that 
$$\gamma_z(0)=z,\qquad \gamma_z(1)=y,\qquad \dist(\gamma_z(t),\pom) \geq \theta\,|\gamma_z(t)-z|
\quad\mbox{for all $t\in [0,1].$}$$

An open set $\Omega \subset \mathbb{R}^{n+1}$ is said to satisfy the \textit{weak local John condition} if there are $\lambda, \theta \in(0,1)$ and $\Lambda \geq 2$ such that the following holds: For every $x \in \Omega$ there is a Borel subset $F \subset B(x, \Lambda \dist(x, \partial \Omega)) \cap \partial \Omega)$ with $\sigma(F) \geq \lambda \sigma(B(x, \Lambda \dist(x, \partial \Omega)) \cap \partial \Omega)$ such that for every $z \in F$ one can find a rectifiable path $\gamma_z:[0,1] \rightarrow \bar{\Omega}$ with length at most $\theta^{-1}|x-z|$ such that
\[
\quad \gamma_z(0)=z, \quad \gamma_z(1)=x, \quad \dist\left(\gamma_z(t), \partial \Omega\right) \geq \theta\left|\gamma_z(t)-z\right| \quad\mbox{ for all }t \in[0,1].
\]

\vv 
Given an $n$-Ahlfors regular set $E \subset \R^{n+1}$, and a surface ball $\Delta_0 = B_0\cap E$, we say that a Borel measure $\mu$ defined on $E$ belongs to \textit{weak-$\mathcal A_\infty(\Delta_0)$} if there are positive constants $C$ and $\theta$ such that for each surface ball $\Delta = B(x,r)\cap E$ with $B(x,2r)\subseteq B_0$, we have
\begin{equation}\label{eq:def_weak_A_infty}
	\mu(F) \leq C \left ( \frac{\mathcal H^n(F)}{\mathcal H^n(\Delta)}\right)^\theta \mu(2\Delta), \quad \mbox{for every Borel set $F\subset \Delta$.}
\end{equation}

\vv

Let us now turn our attention to Sobolev spaces. Let $X$ be a metric space equipped with a doubling measure $\sigma$, that is, there exists a uniform constant $C_\sigma \geq 1$ such that $\sigma(B(x,2r)) \leq C_\sigma \sigma(B(x,r))$, for all $x\in X$ and $r>0$.
For a function $f:X\to \R$, we say that a non-negative function $\nabla_H f:X\to \R$ is a \textit{\haj~gradient} of $f$ if
\begin{equation}
	\label{eq:def_haj_grad}
	|f(x)-f(y)|\leq |x-y| \left(\nabla_H f(x) + \nabla_H f(y)\right) \quad\mbox{for $\sigma$-a.e. $x,y\in X$,}
\end{equation}
and we denote the collection of all \haj~gradients of $f$ by $D(f)$. {Note that the \haj~gradient of a function is not uniquely defined. }\color{black}
For $p\in(0,+\infty]$, we define the \textit{homogeneous \haj-Sobolev} space $\dot M^{1,p}(X)$ as the space of functions $f$ that have a \haj~ gradient $\nabla_H f \in D(f)\cap L^p(X)$, and the \textit{inhomogeneous \haj-Sobolev} space $M^{1,p}(X)$ as $M^{1,p}(X) := \dot M^{1,p}(X) \cap L^p(X)$. These spaces were introduced by \haj~ in \cite{H}.
We define the following norm and seminorm (quasinorm if $p<1$) 
\begin{equation}
	\Vert f \Vert_{M^{1,p}} := 
		\Vert f \Vert_{L^p(X)} + \inf_{g \in D(f)} \Vert g \Vert_{L^p(X)},
\end{equation}
\begin{equation}
	\Vert f \Vert_{\dot M^{1,p}} := 
	\inf_{g \in D(f)} \Vert g \Vert_{L^p(X)}.
\end{equation}
We will often abuse the notation and write $\Vert \nabla_H f\Vert_{L^p}$ to denote $\inf_{g \in D(f)} \Vert g \Vert_{L^p(X)}$.
{We also remark that in the Euclidean case, $M^{1,1}(\R^{n+1})$ coincides with a Hardy-Sobolev space. For more details, see \cite{GMT}.}
\vv

Let $\Omega\subset\R^{n+1}$ be an open set and  set $\sigma:=\HH^n|_{\partial\Omega}$ to be its surface measure. For $\alpha>0$ and $x \in \pom$, we define the {\it cone with vertex $x$ and aperture $\alpha>0$} by
\begin{equation}\label{eqconealpha}
	\gamma_\alpha(x)=\{ y \in \Omega: |x-y|<(1+\alpha)\dist(y, \pom)\},
\end{equation}
the \textit{$R$-truncated cone with vertex $x$ and aperture $\alpha$} by $\gamma_{\alpha,R}(x) = \gamma_\alpha(x) \cap B(x,R)$, and the {\it non-tangential maximal operator} of a measurable function $u:\Omega \to \R$ by
\begin{equation}\label{eqNalpha}
	N_\alpha(u)(x):=\sup_{y \in \gamma_\alpha (x)} |u(y)|, \,\,x\in \pom.
\end{equation}
We also may define the \textit{truncated non-tangential maximal operator} $N_{\alpha,R}$ analogously.
If $\pom$ is Ahlfors regular, then $\| N_\alpha(u)\|_{L^p(\sigma)} \approx_{\alpha,\beta} \| N_\beta(u)\|_{L^p(\sigma)}$ {for all $\alpha,\beta>0$} and so, from now on, we will only write $N$ dropping the dependence on the aperture (see \cite[Proposition 2.2]{HMT} for a proof of this fact). 
Following \cite{KP1}, we also introduce the \textit{modified non-tangential maximal function} for any function $u$ in $L^2_{\operatorname{loc}}(\pom)$ as 
\begin{equation}
	\label{eq:def_modified_nontangential_op}
	\wt N_{\alpha,c}  u(x):= \sup_{y \in \gamma_\alpha(x)} \left (  \fint_{B\left(y, c \dist(x,\pom)\right)} |u(z)|^2\, dz \right)^{1/2}, \quad x \in \pom,
\end{equation}
and, similarly, the \textit{truncated modified non-tangential maximal function} $\wt N_{\alpha,c,R}$. Again, we will omit the dependence on the aperture.
\color{black}
In some parts of the sequel we will consider divergence form elliptic operators $\mathcal L:= \operatorname{div} A\nabla$ where $A$ is always a real valued, not necessarily symmetric $(n+1)\times (n+1)$ matrix verifying the strong ellipticity conditions
\begin{equation}
	\label{eq:ellipticity_conditions_operator}
	\lambda |\xi|^2 \leq \sum_{i,j=1}^{n+1} A_{ij}(x) \xi_i \xi_j, \quad \Vert A \Vert_{L^\infty(\Omega)}\leq \frac 1 \lambda, \quad x\in \Omega, \quad \xi \in \R^{n+1} 
\end{equation}
for some $\lambda \in (0,1)$, the \textit{ellipticity constant} of $\mathcal L$. We will denote by $\mathcal L^*$ the operator associated to the transpose matrix $A^*$. 
\color{black}

\vv

We are now ready to state the definitions of solvability of the $L^p$-Dirichlet and the $M^{1,p}$-regularity problem in $\Omega$ {(with $n$-Ahlfors regular boundary)} for $p \in (0, +\infty)$:
\begin{itemize}
	\item In a domain $\Omega  \subset \R^{n+1}$, we say that {\it the Dirichlet problem is solvable in $L^p$} for the {operator $\mathcal L$}  (we write $(D^{\mathcal L}_{p})$ is solvable) if there exists some constant $C_{D_p}>0$  such that, for any $f\in C_c(\pom)$, the solution $u:\Omega\to\R$ of the continuous
	Dirichlet problem for $\mathcal L$ in $\Omega$ with boundary data $f$ satisfies
	\begin{equation}\label{eq:main-est-Dirichlet}
		\| N(u)\|_{L^p(\sigma)} \leq C_{D_p}\|f\|_{L^p(\sigma)}.
	\end{equation}
	
	\item In a domain $\Omega \subset \R^{n+1}$, we say that {\it the regularity problem is solvable in $M^{1,p}$} for the {operator $\mathcal L$}  (we write $(R^{\mathcal L}_{p})$ is solvable) if  there exists some constant $C_{R_p}>0$  such that, for any  compactly supported Lipschitz function
	$f:\pom\to\R$, the solution $u:\Omega\to\R$ of the continuous
	Dirichlet problem for $\mathcal L$ in $\Omega$ with boundary data $f$ satisfies
	\begin{equation}\label{eq:main-est-reg}
		\| \wt N(\nabla u)\|_{L^p(\sigma)} \leq \begin{cases} C_{R_p}\|f\|_{\dot M^{1,p}(\sigma)}, \quad \mbox{if $\Omega$ is bounded or $\pom$ is unbounded} \\
			C_{R_p}\|f\|_{M^{1,p}(\sigma)}, \quad \mbox{if $\Omega$ is unbounded and $\pom$ is bounded}.
		\end{cases}
	\end{equation}
	Here we denoted $M^{1,p}(\sigma)$ and $\dot M^{1,p}(\sigma)$
	to be the Haj\l asz  Sobolev spaces defined on the metric measure space $(\pom, \sigma)$.
	 
	\item In a domain $\Omega \subset \R^{n+1}$, we say that {\it the regularity problem is solvable in weak $L^1$} for the {operator $\mathcal L$}  (we write $(R^{\mathcal L}_{1,\infty})$ is solvable) if  there exists some constant $C_{R_p}>0$  such that, for any  compactly supported Lipschitz function
	$f:\pom\to\R$, the solution $u:\Omega\to\R$ of the continuous
	Dirichlet problem for $\mathcal L$ in $\Omega$ with boundary data $f$ satisfies
	\begin{equation}\label{eq:main-est-reg-weakL1}
		\| \wt N(\nabla u)\|_{L^{1,\infty}(\sigma)} \leq \begin{cases} C\|f\|_{\dot M^{1,1}(\sigma)}, \quad \mbox{if $\Omega$ is bounded or $\pom$ is unbounded} \\
		C\|f\|_{M^{1,1}(\sigma)}, \quad \mbox{if $\Omega$ is unbounded and $\pom$ is bounded}.
		\end{cases}
	\end{equation}
\end{itemize}
Since we will focus on the harmonic case on most of the sequel, we will omit the operator from the notation when writing $(R_p)$ and $(D_{p'})$. 
\vv

We define the weak normal derivative of  a harmonic function $u$ in $\Omega$ as the functional $F \in \Lip(\pom)^*$ satisfying
\[
\int_{\Omega} \nabla u \nabla \wt \phi \, dm = F(\phi),  
\]
{for all $\phi$ Lipschitz with compact support in $\pom$ and any Lipschitz extension $\wt \phi$}.
Note that this is well defined as if $\wt \phi_1$ and $\wt \phi_2$ are two different Lipschitz extension of $\phi$ to $\Omega$, then its difference satisfies
\[
\int_\Omega \nabla u \nabla(\wt \phi_1 - \wt \phi_2) \, dm = 0
\]
by the harmonicity of $u$.

\vv

\subsection{Main results}
In the case of bounded domains with uniformly rectifiable boundary, we show the following
one-sided Rellich inequality.
\begin{theorem}
	\label{thm:URimpliesRellich}
	Let $\Omega \subset \mathbb R^{n+1}$ be a bounded corkscrew domain with uniformly $n$-rectifiable boundary. Then, for every  function $f\in\Lip(\pom)$, the solution $u_f$ of the Dirichlet problem for the Laplacian  with boundary data $f$ satisfies
	\begin{equation}
	\label{eq:rellich_ineq}
	\Vert \partial_\nu u_f \Vert_{\mathcal M} \lesssim \Vert \nabla_H f \Vert_{L^1(\sigma)}.
	\end{equation}
\end{theorem}

Conversely, we have the following free boundary result in terms of the weak-no-boxes condition (WNB) (see Section \ref{section:WNB} for the definition):
\begin{theorem}
	\label{thm:RellichimpliesWNB}
	Let {$\Omega \subset \R^{n+1}$ be a corkscrew domain with $n$-Ahlfors regular boundary.}
	If \eqref{eq:rellich_ineq} holds for all $f\in\Lip(\pom)$ {and $u_f$ solution to the Dirichlet problem for the Laplacian on $\Omega$}, then $\pom$ satisfies the weak-no-boxes condition.
\end{theorem} 
\begin{remark}
	The WNB (see Section \ref{section:WNB}) was first considered by \cite[page 148]{DS2} as a geometric property of $n$-Ahlfors regular sets $E$ such that its negation implies that we can find Lipschitz functions on $E$ that cannot be well approximated by affine functions at many scales {(see \cite[Definition I.2.47]{DS2} or Section \ref{section:history} for the precise definition of the \textit{WALA} - \textit{weak approximation of Lipschitz functions by affine functions} property).} A geometric characterization of the WALA is  unknown for $n\geq2$.
\end{remark}
In particular, for $1$-Ahlfors regular sets, it is known that the WNB is equivalent to uniform rectifiability (see Theorem \ref{thm:WNBiffUR}). Hence, we obtain the following characterization of uniform $1$-rectifiability for this type of domains.
\begin{corollary}
	Let {$\Omega \subset \R^{n+1}$ be a corkscrew domain with $n$-Ahlfors regular boundary.} If \eqref{eq:rellich_ineq} holds uniformly for all $f\in\Lip(\pom)$ {and $u_f$ solution to the Dirichlet problem for the Laplacian on $\Omega$}, then $\pom$ is uniformly $1$-rectifiable.
\end{corollary}
\begin{remark}
	It is well known that the $L^p$ case (for some $p>1$) of \eqref{eq:rellich_ineq} is equivalent to the density of the harmonic measure of the domain being a weak-$\mathcal A_\infty(\sigma)$ weight (see Appendix \ref{appendix:RellichLpimpliesReverseHolder}). There also exists a more ``natural" endpoint version of the $L^p$ one-sided Rellich inequality in the Hardy $H^1$ space:
	\[
	\Vert \partial_\nu u_f\Vert_{H^1(\sigma)} \lesssim \Vert \nabla_H f \Vert_{L^1(\sigma)}.
	\]
	This inequality is also related to the weak-$\mathcal A_\infty(\sigma)$ property for the density of the harmonic measure, hence falling in the same category as the $L^p$ results. 
	The present result is of a  different nature as it has no relationship to any weak-$\mathcal A_\infty$ weight properties for the density of the harmonic measure or with the connectedness of the domain {(or with the domain having interior big pieces of chord-arc domains)}. This difference can also be observed in the free boundary direction of the proof, {where the authors of \cite{MT2}} showed that $L^p$ bounds for the Poisson kernel imply uniform $n$-rectifiability of the boundary using the theory of singular integrals. Here, the implication is related to the lack of regularity of the Lipschitz functions supported in non-uniformly rectifiable sets, which apparently is of a different nature.
\end{remark}

\vv

We also study the solvability of the regularity problem
$(R_{1,\infty}^\Delta)$ in weak $L^1$.
\begin{theorem}
	\label{thm:URimpliesweakL1}
Let $\Omega \subset \mathbb R^{n+1}$ be a bounded corkscrew domain with uniformly $n$-rectifiable boundary $\pom$. Then, for every  function $f\in\Lip(\pom)\cap M^{1,1}(\pom)$, the solution of the Dirichlet problem for the Laplacian  with boundary data $f$ satisfies
\[
\Vert N(\nabla u_f)\Vert_{L^{1,\infty}(\sigma)} \lesssim \Vert \nabla_H f \Vert_{L^1(\sigma)}.
\]
\end{theorem}
As a corollary, we obtain $(R_p^\Delta)$ solvability for $p<1$ for more general domains than in \cite{GMT}.
\begin{corollary}
	\label{coro:solvabilityRp}
	Let $\Omega \subset \mathbb R^{n+1}$ be a bounded corkscrew domain with uniformly $n$-rectifiable boundary $\pom$. There exists $\epsilon>0$ depending only on $n$ and the Ahlfors regularity constants of $\pom$ such that $(R_q^\Delta)$ is solvable for all $q\in(1-\epsilon,1)$.
\end{corollary}
Note again that these results are different from the ones in \cite{MT} and \cite{GMT} as these do not require any connectivity on the domain (or weak-$\mathcal A_\infty$ properties for the harmonic measure). {Summarizing, for a bounded corkscrew domain $\Omega$ with $n$-Ahlfors regular boundary, we have}
\begin{align*}
	&\mbox{$\pom$ is $n$-UR}
	\mbox{ and $\Omega$ satisfies weak local John}  \iff \omega\in \mbox{weak-}\mathcal A_\infty \iff (R_1^\Delta) \mbox{ holds,} \\
	&\mbox{$\pom$ is $n$-UR}
	\mbox{ and $\Omega$ satisfies weak local John} \not\Leftarrow\mbox{$\pom$ is $n$-UR} \implies (R_{1,\infty}^\Delta) \mbox{ holds.}
\end{align*}
\color{black}
Nonetheless, solvability of $(R_{1,\infty}^\Delta)$ is not trivial as we have a counterexample when $\pom$ is the $4$-corners Cantor set.
\begin{proposition}
	\label{prop:NweakL1inCantor}
	Let $E$ be the $4$-corners Cantor set in $\mathbb R^2$ and $\Omega = B(0,100)\backslash E$. For every $M>0$ {and operator $\mathcal L$ satisfying \eqref{eq:ellipticity_conditions_operator}}, there exists a $1$-Lipschitz function $f$ on $\pom$ {such that the solution $u_f$ to the Dirichlet problem for $\mathcal L$ satisfies}
	\[
	\Vert \wt N(\nabla u_f )\Vert_{L^{1,\infty}(\sigma)} \geq M. \quad 
	\]
\end{proposition}
In fact, we prove the existence of a counterexample of $(R_{1,\infty}^\mathcal L)$ for general divergence form elliptic operators $\mathcal L$. In turn this allows us to prove that in general domains with $n$-Ahlfors regular boundary the Dirichlet problem and the regularity problem are not equivalent.
\begin{proposition}
	\label{prop:DirichetdoesnotimplyRegularity}
	For general elliptic operators $\mathcal L$ satisfying \eqref{eq:ellipticity_conditions_operator}, the 
	solvability of the Dirichlet problem $(D_{p'}^{\mathcal L^*})$ does not imply solvability of the regularity problem $(R_{p}^{\mathcal L})$ (where $p,p'>1$ are H\"older conjugate exponents).
\end{proposition}
Note that in the particular case that $\mathcal L = \Delta$ or $\mathcal L$ is a Dahlberg-Kenig-Pipher operator, then the two problems are equivalent (see \cite{MT} and \cite{MPT} respectively).

{As far as the author knows, Proposition \ref{prop:DirichetdoesnotimplyRegularity}   answers a question of Kenig and Pipher about whether the solvability of Dirichlet problem implies solvability of regularity problem in general domains (see \cite[Section 3]{KP2})}.

\subsection{Historical remarks}
\label{section:history}
Let us provide some historical context regarding the solvability of the regularity problem, its relationship with Rellich estimates and the WALA property. For an extensive account of the history of the solvability of the Dirichlet problem in rough domains for the Laplacian, and more general elliptic operators and systems, we refer the reader to the introduction of \cite{MPT}. For the WALA property, see \cite{AMV} and the book \cite{DS2}. 

The regularity problem was first solved in $C^1$ domains for all $p\in(1,+\infty)$ by {Fabes, Jodeit and Rivi\`ere in \cite{FJR}} by proving the invertibility of the layer potential operators using Fredholm theory. In Lipschitz domains, Fredholm theory is not available and new ideas were needed in order to show invertibility of the layer potentials. {Verchota \cite{V}} used the techniques introduced by Jerison and Kenig \cite{JK} to show the solvability of $(R_p^\Delta)$ for $p\in(1,2]$. The main tool in \cite{JK}, where the solvability of $(D_2^{\mathcal L})$ was proved for elliptic operators in divergence form with smooth and $L^\infty$ coefficients, was the so-called Rellich inequality.  This inequality states that in a Lipschitz domain, the tangential derivative of solution to the Dirichlet problem and its normal derivative are comparable in the sense
\begin{equation*}
	\Vert \nabla_t u \Vert_{L^2(\pom)} \approx \Vert \partial_\nu u \Vert_{L^2(\pom)}  
\end{equation*}
and was obtained via  a clever integration by parts argument.
This inequality also been employed to show the solvability of the Neumann problem in Lipschitz domains for the Laplacian in \cite{JK3}. 

Rellich inequalities have also been used to study the duality between $(R_p^{\mathcal L})$ and $(D_{p'}^{\mathcal L^*})$  for complex-valued operators in $\R^{n+1}_+$ by Hofmann, Kenig, Mayboroda and Pipher in \cite{HKMP} and for ellyptic systems by Auscher and Mourgoglou in \cite{AMo}. 
In bounded corkscrew domains with uniformly $n$-rectifiable boundary, \cite{MT} proved $L^p$ one-sided Rellich inequalities as an intermediate step in proving the implication $(D_{p'}^\Delta) \implies (R_p^\Delta)$. {They were also the first to consider the regularity problem in \haj-Sobolev spaces.} It is also remarkable the result of Mourgoglou, Poggi and Tolsa in \cite{MPT} where they show $(D_{p'}^{\mathcal L^*}) \implies (R_p^{\mathcal L})$ for the class of Dahlberg-Kenig-Pipher operators in bounded corkscrew domains with uniformly $n$-rectifiable boundary. Their approach also relied on a Rellich  type inequality for solutions of a certain Poisson-Dirichlet problem.
Note that for general elliptic operators satisfying \eqref{eq:ellipticity_conditions_operator}, Kenig and Pipher \cite{KP1} showed that $(R_p^{\mathcal L}) \implies (D_{p'}^{\mathcal L^*})$ {for $1<p<\infty$} {in Lipschitz domains  and Mourgoglou and Tolsa \cite{MT} proved it in corkscrew domains with $n$-Ahlfors regular boundary}.

{Other recent developments on the regularity problem under varying assumptions include: Dindos and Kirsch \cite{DK2} on the regularity problem in the Hardy-Sobolev space $HS^1$ for Lipschitz domains; the author, Mourgoglou and Tolsa \cite{GMT} on the regularity problem in $M^{1,1}(\pom)$ for rough domains; Dai, Feneuil and Mayboroda \cite{DFM} on the stability of solvability of the regularity problem under Carleson perturbations of the operator; Dindos, Hofmann and Pipher \cite{DHP} for an alternative (shorter) proof of the results of \cite{MPT} in Lipschitz domains, and Feneuil \cite{Fen} for a different simpler proof of the results of \cite{MPT} in $\mathbb R^{n+1}_+$.} {In the setting of parabolic PDEs, the solvability of the $L^p$ regularity problem has been recently shown by Dindos, Li and Pipher \cite{DLP} for a certain class of operators in Lipschitz cylinders.} {The solvability of the Dirichlet and Neumann problems has also been studied in various other geometric settings: by Hofmann, M. Mitrea, and Taylor \cite{HMT} in $\delta$-regular Semmes-Kenig-Toro domains; by Marín, Martell, D. Mitrea, I. Mitrea, and M. Mitrea \cite{MMMM} in two-sided chord-arc domains where the outer unit normal $\nu$ has small BMO norm; and by Guillén \cite{Gui} in a refinement of the \cite{MMMM} setting that allows the BMO norm of $\nu$ to be large at certain scales }
\vvv

David and Semmes introduced and developed the theory of uniformly rectifiable sets in \cite{DS1,DS2}, formulating a framework under which all ``good" $n$-dimensional singular integral operators are bounded in $L^2$. One property of uniform rectifiable sets is that they satisfy the \textit{Weak approximation of Lipschitz functions by Affine functions} (WALA), that is, for any $1$-Lipschitz function $f$ on  a uniformly $n$-rectifiable $E$ and $\epsilon>0$, the following holds. Let $\mathcal G(\epsilon) \in E\times [0,\diam(E)] \subseteq E \times \mathbb R_+$ be the set of points $x\in E$ and scales $0<r<\diam(E)$ such that there exists an affine function $A_{x,r}$ with $\Vert f - A_{x,r}\Vert_{L^\infty({B(x,r)\cap E})} \leq \epsilon r$. Then the set $E\times [0,\diam(E)] \backslash \mathcal G(\epsilon)$ is Carleson. This property can be regarded as a quantitative version of Rademacher's theorem on differentiability of Lipschitz functions a.e. on rectifiable sets and, in fact, an analogue of the result of Dorronsoro \cite{Do} for Lipschitz functions is true on uniformly rectifiable sets (which is a stronger than the WALA). Azzam, Mourgoglou and Villa \cite{AMV} extended this result and proved an $L^p$ version of Dorronsoro's theorem for Sobolev functions on uniformly $n$-rectifiable sets. Whether the WALA property alone implies uniform $n$-rectifiability in higher dimensions is still an open question. In the one dimensional case, however, it was shown in \cite{DS2} that $1$-Ahlfors regular sets satisfying the WALA also satisfy the WNB, which is equivalent to uniform $1$-rectifiability of the set.

\vvv
\subsection{Strategy of the proofs}
For the proof of Theorem \ref{thm:URimpliesRellich} and Theorem \ref{thm:URimpliesweakL1}, we follow a strategy inspired by the one used in \cite{MT}. The uniform rectifiability of the boundary implies that there exists a corona decomposition of the domain $\Omega$ in Lipschitz subdomain $\Omega_R$ where the Dirichlet and regularity problems are uniformly solvable. This allows us to define an ``almost" harmonic extension $v_f$ of the boundary data $f$, that is a function that is a simple Lipschitz extension on $f$ on the bad set of the corona decomposition but is harmonic on every Lipschitz subdomain. This extension together with the nice properties of the subdomains is crucial in showing \eqref{eq:rellich_ineq}, and then the boundedness of singular integrals on uniformly $n$-rectifiable implies Theorem \ref{thm:URimpliesweakL1}.

To prove Theorem \ref{thm:RellichimpliesWNB}, we show that if the boundary does not satisfy WNB, then for all $C>0$ there exists a function for which \eqref{eq:rellich_ineq} with constant $C$ {fails}. The presence of many ``boxes" allows us to construct a family $(f_i)_i$ of ``independent" Lipschitz functions such that linear combinations of $f_i$ have bounded Lipschitz constant but each $f_i$ has large $\Vert \partial_\nu f_i \Vert_{\mathcal M}$. A priori, the linear combinations of these $\partial_\nu f_i$ can have a lot of cancellation but we circumvent this by taking a sum with random signs. {Using Khintchine's inequality we show, in expectation, that \eqref{eq:rellich_ineq} is not satisfied for a random sum.} 

Proposition \ref{prop:NweakL1inCantor} is proved with a similar argument to the one used to show Theorem \ref{thm:RellichimpliesWNB}, but the weak $L^1$ boundedness of the maximal nontangential operator requires a more careful analysis that we have only carried out in the complementary of the $4$-corners Cantor set. Specifically, we use that this domain has the very nice property that every dyadic cube of the boundary contains a ``box". Thus the bad functions we construct are uniformly bad at every point and scale.
Finally, Proposition \ref{prop:DirichetdoesnotimplyRegularity} uses the construction of \cite{DM} of an elliptic operator on the complementary of the $4$-corners Cantor set such that its elliptic measure is comparable to surface measure (hence the Dirichlet problem is solvable for all $p>1$). If the regularity problem were solvable for some $p>1$ for the adjoint operator, by extrapolation this would imply the solvability of the regularity in weak $L^1$, but this contradicts Proposition \ref{prop:NweakL1inCantor}.

\section{Preliminaries}
\textbf{Notation:} We distinguish between $n$-dimensional dyadic cubes of the boundary and $(n+1)$-dimensional cubes inside the domain by using a hat ($\hat{\ }$) over the latter.


\subsection{Dyadic lattice}
Given an $n$-Ahlfors-regular measure $\mu$ in $\mathbb{R}^{n+1}$, we consider the dyadic lattice of ``cubes'' built by David and Semmes in \cite[Chapter 3 of Part I]{DS2}. The properties satisfied by $\mathcal{D}_\mu$ are the following. Assume first, for simplicity, that $\mathrm{diam}(\mathrm{supp}\,\mu) = \infty$. Then for each $j \in \mathbb{Z}$ there exists a family $\mathcal{D}_{\mu}^{j}$ of Borel subsets of $\mathrm{supp}\,\mu$ (the dyadic cubes of the $j$-th generation) such that:

\begin{itemize}
	\item[(a)] each $\mathcal{D}_{\mu}^{j}$ is a partition of $\mathrm{supp}\,\mu$, i.e., $\mathrm{supp}\,\mu = \bigcup_{Q \in \mathcal{D}_{\mu}^{j}} Q$ and $Q \cap Q' = \varnothing$ whenever $Q, Q' \in \mathcal{D}_{\mu}^{j}$ and $Q \neq Q'$;
	\item[(b)] if $Q \in \mathcal{D}_{\mu}^{j}$ and $Q' \in \mathcal{D}_{\mu}^k$ with $k \leq j$, then either $Q \subset Q'$ or $Q \cap Q' = \varnothing$;
	\item[(c)] for all $j \in \mathbb{Z}$ and $Q \in \mathcal{D}_{\mu}^{j}$, we have $2^{-j} \lesssim \mathrm{diam}(Q) \leq 2^{-j}$ and $\mu(Q) \approx 2^{-jn}$;
	\item[(d)] there exists $C > 0$ such that, for all $j \in \mathbb{Z}$, $Q \in \mathcal{D}_{\mu}^{j}$, and $0 < \tau < 1$,
	\begin{align*}
		&\mu\big(\{x \in Q : \mathrm{dist}(x, \mathrm{supp}\,\mu \setminus Q) \leq \tau 2^{-j} \}\big)
		+\\ &\mu\big(\{x \in \mathrm{supp}\,\mu \setminus Q : \mathrm{dist}(x, Q) \leq \tau 2^{-j} \}\big) \leq C \tau^{1/C} 2^{-jn}.
	\end{align*}
\end{itemize}

This property is usually called the \emph{small boundaries condition}. 

We set $\mathcal{D}_\mu := \bigcup_{j \in \mathbb{Z}} \mathcal{D}_{\mu}^{j}$.

In case that $\operatorname{diam}(\operatorname{supp} \mu) < \infty$, the families $\mathcal{D}_{\mu}^{j}$ are only defined for $j \geq j_0$, with $2^{-j_0} \approx \operatorname{diam}(\operatorname{supp} \mu)$, and the same properties above hold for $\mathcal{D}_\mu := \bigcup_{j \geq j_0} \mathcal{D}_{\mu}^{j}$.

Given a cube $Q \in \mathcal{D}_{\mu}^{j}$, we say that its side length is $2^{-j}$, and we denote it by $\ell(Q)$. Notice that 
\[
\operatorname{diam}(Q) \leq \ell(Q).
\]

For $\lambda > 1$, we write
$$ \lambda Q = \left\{ x \in \operatorname{supp} \mu : \operatorname{dist}(x, Q) \leq (\lambda - 1) \ell(Q) \right\}. $$


\vvv
\subsection{Whitney decomposition associated to the dyadic decomposition of the boundary}
Let $\Omega\subset\mathbb R^{n+1}$ be a corkscrew domain with $n$-Ahlfors regular boundary $\pom$ and $\sigma= \HH^{n}|_{\pom}$.
We consider the following Whitney decomposition of $\Omega$ (assuming $\Omega \neq \mathbb{R}^{n+1}$): we have a family $\mathcal{W}(\Omega)$ of dyadic cubes in $\mathbb{R}^{n+1}$ with disjoint interiors such that $\bigcup_{\hat P \in \mathcal{W}(\Omega)} \hat P = \Omega$, and moreover there are some constants $\Lambda > 20$ and $D_0 \geq 1$ such that the following holds for every $\hat P \in \mathcal{W}(\Omega)$:
\begin{itemize}
	\item[(i)] $10\hat P \subset \Omega$;
	\item[(ii)] $\Lambda \hat P \cap \partial \Omega \neq \varnothing$;
	\item[(iii)] there are at most $D_0$ cubes $\hat P' \in \mathcal{W}(\Omega)$ such that $10\hat P \cap 10\hat P' \neq \varnothing$. Further, for such cubes $\hat P'$, we have $\frac{1}{2} \ell(\hat P') \leq \ell(\hat P) \leq 2\ell(\hat P')$.
\end{itemize}

From the properties (i) and (ii) it is clear that $\text{dist}(\hat P, \partial \Omega) \approx \ell(\hat P)$. We assume that the Whitney cubes are small enough so that
\begin{equation*} 
	\text{diam}(\hat P) < \frac{1}{20} \, \text{dist}(\hat P, \partial \Omega).
\end{equation*}

The arguments to construct a Whitney decomposition satisfying the properties above are standard.

Suppose that $\partial \Omega$ is $n$-AD-regular and consider the dyadic lattice $\mathcal{D}_\sigma$ defined above. Then, for each Whitney $\hat P \in \mathcal{W}(\Omega)$ there is some cube $Q \in \mathcal{D}_\sigma$ such that $\ell(Q) = \ell(\hat P)$ and $\text{dist}(\hat P, Q) \approx \ell(Q)$, with the implicit constant depending on the parameters of $\mathcal{D}_\sigma$ and on the Whitney decomposition. We denote this by $Q = b(\hat P)$ (``b'' stands for ``boundary''). Conversely, given $Q \in \mathcal{D}_\sigma$, we let
\begin{equation*} 
	w(Q) = \bigcup_{\hat P \in \mathcal{W}(\Omega) : Q = b(\hat P)} \hat P.
\end{equation*}

It is immediate to check that $w(Q)$ is made up of at most of a uniformly bounded number of cubes $\hat P$, but it may happen that $w(Q) = \varnothing$.

\vvv

\subsection{Layer potentials}
\label{section:layer_potentials}
Let $\mathcal E$ be the fundamental solution for the Laplacian.
We denote by $S$ the single layer potential associated to the Laplacian
\[
S f(x) = \int_{\pom} \mathcal E(x-y) f(y)\, d\sigma(y), \quad x\in\Omega
\] and by $D$ the double layer potential
\[
D f(x) = \int_{\pom} \nu(y) \cdot \nabla_y\mathcal E(x-y) f(y)\, d\sigma(y), \quad x\in\Omega
\]
where $\nu$ stands for the measure theoretic outer unit normal of $\Omega$.

These operators have the following relationship with the  solutions of the Dirichlet problem:

\begin{lemma}[Lemma 6.1 in [MT{]}]
	Let $\Omega \subset \mathbb{R}^{n+1}$ be a domain with uniformly $n$-rectifiable boundary $\pom$. Given $f \in \operatorname{Lip}(\partial \Omega)$, denote by $u$ the solution of the Dirichlet problem in $\Omega$ with boundary data $f$. Then we have
	$$
	u(x)={D}(u|_{\partial \Omega})(x)-{S}(\partial_\nu u|_{\partial \Omega})(x) \quad \text { for all } x \in \Omega,
	$$
	where ${D}$ and ${S}$ denote the double and single layer potentials for $\Omega$, respectively.
\end{lemma}

These operators are also intimately related to the Riesz transforms $\mathcal R$.
By the results of \cite{DS1}, we know that when $E \subset \R^{n+1}$ is an uniformly $n$-rectifiable set, then the Riesz transforms $\mathcal R$ are bounded on $L^2(E)$. The converse result is also true and is due to \cite{NTV}. Hence,

\begin{proposition}[Prop 3.20 in \cite{HMT}]
	\label{prop:bdd_gradient_single_layer}
	Let $\Omega \subset \R^{n+1}$ be a domain with uniformly $n$-rectifiable boundary. Then,
	\[
	\Vert N(\nabla S f) \Vert_{L^{1,\infty}(\sigma)} \lesssim \Vert f \Vert_{L^1(\sigma)}.
	\]
	We also have
	\[
	\Vert N(\nabla S \mu) \Vert_{L^{1,\infty}(\sigma)} \lesssim \Vert \mu \Vert_{\mathcal M}
	\]
	for Radon measures $\mu$.
\end{proposition}
The second part of the proposition is not proved in \cite{HMT} but the only changes needed in the proof are: \cite[Prop 3.19]{HMT} should be changed by \cite[Theorem 2.21]{T} and \cite[display (2.1.4)]{HMT} by \cite[Theorem 2.5]{T}.

{In a similar manner, we have the boundedness for the gradient of the double layer potential as a consequence of \cite[Proposition 3.37]{HMT} (although there it is only written for $p>1$ and in terms of tangential derivatives)}.
\begin{proposition}[Prop 3.37 in \cite{HMT}]
	\label{prop:bdd_double_layer_potential}
	Let $\Omega \subset \R^{n+1}$ be a domain with uniformly $n$-rectifiable boundary, $f\in \Lip(\pom)\cap M^{1,1}(\sigma)$. Then,
	\[
	\Vert N(\partial_j D f) \Vert_{L^{1,\infty}(\sigma)} \lesssim  \Vert \nabla_H f \Vert_{L^1(\sigma)}
	\]
	for all $j$ in $\{1,\hdots, n+1\}$.
\end{proposition}
{Note that the tangential derivatives $\partial_{t,j,k} f$ introduced by \cite{HMT} are bounded above by $\nabla_H f$ $\sigma$-a.e. (see \cite[Section 2.4]{MT}) which yields the previous Proposition from \cite[Proposition 3.37]{HMT}.}
%
%
%

\vvv
\subsection{Rellich inequality in Lipschitz domains}

The following proposition will be necessary in the sequel.
\begin{theorem}[Rellich inequality in Hardy spaces]
	\label{thm:Rellich_ineq_Lipschitz}
	Let $\Omega\subset\R^{n+1}$ be a bounded Lipschitz domain, and $u$ be a harmonic function in $\Omega$ with boundary value $f\in\operatorname{Lip}(\pom)$. Then, we have
	\[
	\Vert \partial_\nu u \Vert_{H^1(\pom)} \lesssim 
	\Vert \nabla_t f \Vert_{ H^1 (\pom)}
	\]
	where $\partial_\nu u$ is the outer normal derivative of $u$ on $\pom$.
\end{theorem}
This is essentially a consequence of the solvability of the regularity problem $(R_1^\Delta)$ on Lipschitz domains (see Theorems 4.3 and 4.12 in \cite{DK}). 

\vvv
\subsection{Corona decomposition of the domain $\Omega$ in starlike Lipschitz domains with small Lipschitz constant}
\begin{theorem}
	\label{thm:corona_decomposition}
	Let $\Omega \subset \mathbb R^{n+1}$ be a bounded corkscrew domain with uniformly $n$-rectifiable boundary $\pom$.
 	There exists a decomposition of $\Omega$ in disjoint sets
\[
\Omega = H \cup \bigcup_{R\in \Top} \Omega_R
\]
satisfying
\begin{itemize}
	\item[(i)] $\Top$ is a subset of $\Dsigma$ satisfying the Carleson packing condition $\sum_{R\in\Top, R\subset Q} \sigma(R) \lesssim \sigma(Q)$ for all $Q\in\Dsigma$,
	\item[(ii)] $\Omega_R$ is a Lipschitz domain with bounded Lipschitz constant for all $R\in\Top$,
	\item[(iii)] there is a family $\mathscr H \subset \Dsigma$ satisfying $H \subset \bigcup_{Q \in \mathscr H} w(Q)$ and the packing condition $\sum_{Q\in\mathscr H, Q \subset S} \sigma(Q) \lesssim \sigma(S)$ for all $S\in\Dsigma$,
	\item[(iv)] {for each Whitney cube $\hat Q \in \WW$ there exists at most a uniformly bounded number of $R\in\Top$ satisfying $\Omega_R \cap \hat Q \neq \varnothing$.} \color{black} 
	
\end{itemize}
\end{theorem}
The proof is in \cite[Section 3]{MT}. Note that the authors prove it for $n\geq2$ but their arguments work as well for $n=1$.

\vvv
\subsection{Atomic decompositions of $M^{1,p}$ spaces for $p\leq1$}
Let $(X, \,d(\cdot,\cdot),\, \sigma)$ be a metric space $X$ equipped with an $n$-Ahlfors regular measure $\sigma$.

\begin{definition}[$(\infty,t,p)$-atom]
	\label{def:atom}
	A $\sigma$-measurable function $a$ is an $(\infty,t,p)$-atom for $p \in (0,+\infty)$ and $t\in(p,+\infty]$ if there exists a ball $B$ such that
	\begin{itemize}
		\item[(a)] $\supp a \subset B$,
		\item[(b)] $\Vert \nabla_H a \Vert_{L^t(X)} \leq \sigma(B)^{\frac 1 t - \frac1 p}$.
	\end{itemize}
	
	\begin{theorem}[Theorem 3.4 \cite{GMT}]
		\label{thm:atomic_decomposition}
		Let $p\in (0,1]$, $f\in \dot M^{1,p}(X)\cap\operatorname{Lip}(X)$. There exists a constant $C>0$, a sequence $(a_j)_j$ of {Lipschitz} $(\infty,\infty,p)$-atoms, and a sequence $(\lambda_j)_j$ of real numbers such that 
		\[
		f = \sum_j \lambda_j a_j \quad\mbox{in $\dot M^{1,p}(X)$ }\color{black}\quad\text{ and }\quad  \sum_j |\lambda_j|^p  \leq C  \Vert \nabla_H f \Vert_p^p.
		\] 
	\end{theorem}
	
\end{definition}

\vvv
\subsection{Estimates about elliptic measure}

Assume that $\pom$ is $n$-Ahlfors regular, thus the continuous Dirichlet problem for an elliptic operator $\mathcal L$ satisfying \eqref{eq:ellipticity_conditions_operator} is solvable. By the maximum principle and the Riesz Representation Theorem, there exists a family of probability measures $\{\omega^x_{\mathcal L}\}_{x\in\Omega}$ on $\pom$ so that for each $f \in C_c(\pom)$  and each $x \in \Omega$, the solution $u$ to the continuous Dirichlet problem with data $f$ satisfies $u(x) = \int_{\pom} f(\xi) \, d\omega^x_{\mathcal L}(\xi)$. We call $\omega^x_{\mathcal L}$ the \textit{elliptic measure with pole at $x$} (or harmonic measure if $\mathcal L = \Delta$).

We present two very important lemmas  in the study of the {elliptic} measure $\omega$ of a domain with $n$-Ahlfors regular boundary.
\begin{lemma}[Lemma 2.16 in \cite{MPT}]
	\label{lemma:Bourgain}
	Let $\Omega \subsetneq \mathbb{R}^{n+1}$ be an open set with $n$-Ahlfors regular boundary $\pom$. Then there exists $c>0$ depending only on $n$, the $n$-Ahlfors regularity constant of $\pom$, and the ellipticity constant of $\mathcal L$, such that for any $\xi \in \partial \Omega$ and $r \in(0, \operatorname{diam}(\partial \Omega) / 2]$, we have that $\omega^x_{\mathcal L}(B(\xi, 2 r) \cap \partial \Omega) \geq c$, for all $x \in \Omega \cap B(\xi, r)$.
\end{lemma}
\begin{lemma}[Lemma 2.19 in \cite{MPT}]
	\label{lemma:Holderdecay}
	Let $\Omega\subset \mathbb R^{n+1}$ be an open set with $n$-Ahlfors regular boundary $\pom$ and $B = B(\xi, r)$ be a ball centered on $\pom$ with $r\leq \diam B$. Let $u$ be a non-negative solution of $\mathcal Lu=0$ on $\Omega$, vanishing continuously in $\pom\cap B$. There exists $\alpha>0$ such that
	\[
	u(x) \lesssim \left(\frac{|x-\xi|}{r}\right)^\alpha \sup_{B\cap\Omega} u \quad \mbox{for $x\in\Omega\cap B$.}
	\]
\end{lemma}
\vv
In the case $\Omega$ is a \textit{uniform domain}, that is, every pair of points can be connected with a Harnack chain, elliptic measures have better properties that will be useful in the {counterexample} in the complementary of the $4$-corners Cantor set.
\begin{theorem}
	\label{thm:comparabilityomegaGreen}
	Let $\Omega\subset \mathbb R^{n+1}$ be a uniform domain with $n$-Ahlfors regular boundary and $B$ be a ball centered on $\pom$ with $r(B)\leq \diam B$. For $x \in \Omega \backslash B$,
	\[
	\omega_{\mathcal L}^x(B) \approx r^{n-1} G(x, x_B)
	\]
	where $x_B\in\Omega$ is a corkscrew point for the ball $B$ and $G$ is the Green function with pole at $x$.
\end{theorem}

\begin{theorem}[Boundary Harnack principle] 
	\label{thm:bdryHarnack}
	
	Let $\Omega$ be a uniform domain with $n$-Ahlfors regular boundary and $B$ be a ball centered on $\pom$ with $r(B)\leq \diam B$. Let $u,v \geq 0$ be solutions of $\mathcal Lu=0$ on $\Omega$, vanishing continuously in $\pom\cap B$ and such that $u(x_B) = v(x_B)$ where $x_B\in\Omega$ is a corkscrew point for the ball $B$.
	Then
	\[
	u(x) \approx v(x), \quad \mbox{for $x\in$} \frac 1 2 B \cap \Omega
	\]
	with comparability constant independent of $u$ and $v$.
	
\end{theorem}

\vvv
\subsection{Probability lemmas}
We introduce some probability results that will be necessary to prove that the one-sided Rellich inequality \eqref{eq:rellich_ineq} implies WNB.

\begin{proposition}[Khintchine's Inequality]
	\label{prop:Khintchine}
	Let \( (\epsilon_n)_{n=1}^{\infty} \) be a sequence of independent random variables, where each \( \epsilon_n \) takes values \( \pm 1 \) with equal probability. Then, there exists a constant \( C > 0 \) such that for any sequence of positive numbers \( (a_n)_{n=1}^{\infty} \in \ell^2\), we have
	\[
	C^{-1} \left( \sum_{n=1}^{\infty} a_n^2 \right)^{1/2} \leq  \mathbb{E} \left[ \left| \sum_{n=1}^{\infty} a_n \epsilon_n \right| \right]  \leq C \left( \sum_{n=1}^{\infty} a_n^2 \right)^{1/2}.
	\]
\end{proposition}
\begin{proposition}[Paley-Zygmund inequality]
	\label{prop:paleyzygmund}
	Let $Z$ be a nonnegative random variable with finite first and second moments. Then
	\[
	\mathbb P\left(Z \geq \frac{\mathbb E (Z)}{2}\right) \geq c \frac{\mathbb E(Z)^2}{\mathbb E(Z^2)}
	\]
	where $c$ is an absolute constant independet of $Z$.
	In the particular case $Z = |\sum_n \epsilon_n a_n|$, with $(a_n)_n \in \ell^2$ and $\epsilon_n$ are independent random variables that take values $\pm1$ with equal probability, we have
	\[
	\mathbb P\left(Z \geq \frac{\mathbb E (Z)}{2}\right) \geq c'.
	\]
\end{proposition}

\section{Proof of {Theorems} \ref{thm:URimpliesRellich} and \ref{thm:URimpliesweakL1}}
In this section we will prove Theorems \ref{thm:URimpliesRellich} and \ref{thm:URimpliesweakL1} following a strategy similar to the one implemented in \cite{MT}. {We will show that for $f$ Lipschitz in $\pom$, the weak normal derivative  of the Dirichlet solution $u_f$ with boundary data $f$ is a bounded functional on $C_c(\pom)$. To do so, our main tool will be the Corona Decomposition \ref{thm:corona_decomposition} developed by \cite{MT}.} Assume $\Omega\subset \mathbb R^{n+1}$ is a bounded corkscrew domain with uniformly $n$-rectifiable boundary.
Let $B_0$ be a ball centered on $\pom$ with $r(B_0)\leq \diam (\pom)$ and $f$ a Lipschitz function on $\pom$ supported in $B_0\cap \pom$, and partition $\Omega$ into $ H \cup \bigcup_{R\in\Top} \Omega_R$ using the corona decomposition of Theorem \ref{thm:corona_decomposition}.
\subsection{Almost harmonic extension of Lipschitz functions on the boundary}

Let $\wt f$ be a Lipschitz extension of $f$ to $\overline\Omega$ satisfying $\wt f \equiv 0$ on $(2B_0)^c \cap \overline\Omega$ with $\Lip(\wt f) \lesssim \Lip(f)$, which exists by the Whitney extension theorem.
As in \cite{MT}, we define the \textit{almost harmonic extension} $v=v(f)$ of $f$ by
\[
v:= \begin{cases}
	\wt f & \mbox{in } \overline\Omega\backslash \bigcup_{R\in\Top} \Omega_R\\
	v_R & \mbox{in each $\Omega_R$, for $R\in\Top$,}
\end{cases}
\]
where $v_R$ is the solution of the Dirichlet problem in the Lipschitz subdomain  $\Omega_R$ with boundary data $\wt f|_{\pom_R}$.
Note that
\[
\supp v \subset \left(\overline\Omega \cap 2 B_0\right) \cup \bigcup_{R\in\Top, \,\,\Omega_R \cap 2B_0 \neq \varnothing} \Omega_R 
\]
and
\[
\Vert \nabla v \Vert^2_{L^2(\Omega)} \lesssim \Lip(f)^2 m(\Omega).
\]
See \cite[Lemma 4.3]{MT} for a proof of the second statement.

\subsection{Estimates on the almost harmonic extension}
Let $w \in  W^{1,2}(\Omega) \cap \operatorname{Lip}(\pom)$ be a harmonic function in $\Omega$.
Then, we trivially have
\[
\left|\int_\Omega (\nabla w, \nabla v)\, dm \right|< +\infty.
\]
since both $w$ and $v$ are in $W^{1,2}(\Omega)$. We will show the following.
\begin{lemma} We have
	\label{lemma:estimate_nablaw_nablav}
	\[
	\left| \int_\Omega (\nabla w, \nabla v)\, dm \right| \lesssim \Lip(f) \cdot \sigma(B_0) \cdot \Vert w \Vert_{L^\infty(\pom)}.
	\]
\end{lemma}
\begin{proof}
	We start by recalling that $\supp v \subset 2 B_0 \cup \bigcup_{R\in\Top,\,\, \Omega_R \cap 2B_0 \neq \varnothing} \Omega_R$.
	Hence, we can bound
	\begin{align*}
		\left| \int_\Omega (\nabla w, \nabla v)\, dm \right| \leq
		&\sum_{R\in\Top,\,\, \Omega_R \cap 2B_0 \neq \varnothing} \left|\int_{\Omega_R}(\nabla w, \nabla v)\, dm\right| \\+ &\sum_{\hat Q\in\mathcal W(\Omega),\,\,Q\cap H\cap 2B_0 \neq \varnothing} \int_{\hat Q\cap H} \left|(\nabla w, \nabla v)\right| \, dm
	\end{align*}
	where $H$ is the part of $\Omega$ not covered by the subdomains $\Omega_R$.

	
	Using that the domains $\Omega_R$ are Lipschitz, Green's formula and that $v$ is harmonic inside every subdomain $\Omega_R$, we have
	\[
	\int_{\Omega_R} (\nabla w, \nabla v)\, dm = \int_{\pom_R} w \, \partial_{\nu_R} v\, d\HH^n|_{\pom_R}
	\]
	where $\nu_R$ is the outer unit normal to $\pom_R$.
	From Theorem \ref{thm:Rellich_ineq_Lipschitz} and the duality $H^1$-$BMO$ we obtain
	\[
	\left|\int_{\pom_R} w \partial_{\nu_R} v\, d\HH^n|_{\pom_R}\right| \lesssim \Vert \nabla_t \wt f\Vert_{H^1(\pom_R)} \Vert w \Vert_{BMO(\pom_R)}
	\]
	where $\nabla_t$ stands for tangential derivative on $\Omega_R$ and we are taking into account that $\nabla_t v = \nabla_t \wt f$ for $\HH^n$-a.e. $x\in \pom_R$.
	
	Now, we want to control
	\[
	\sum_{R \in \Top} \Vert \nabla_t \wt f \Vert_{H^1(\pom_R)}.
	\]
	Observe that $|\nabla_t \wt f| \leq \Lip(\wt f) \lesssim \Lip(f)$ $\HH^n$-a.e $x\in\pom_R$. 
	Furthermore, $\supp (\nabla_t \wt f|_{\pom_R}) \subset 2B_0 \cap \pom_R$ and $\int_{\pom_R} \nabla_t \wt f \, d\HH^n|_{\pom_R}=0$. Hence, $\nabla_t \wt f|_{\pom_R}$ is a multiple of a ({vector-valued}) Lipschitz $H^1$-atom, and has $H^1(\pom_R)$ norm at most $C\Lip(f) \cdot \HH^n(2B_0\cap\pom_R)$.
	Thus, using that every Whitney cube $\hat Q\in\WW$ intersects at most a bounded number of domains $\Omega_R$, that the cubes $\hat Q\in\mathcal W$ that intersect $\pom_R$ also intersect $H$ and hence must satisfy a Carleson packing condition, {that is $\sum_{\hat Q \in \WW, \, \hat Q \cap H \neq \varnothing} \ell(\hat Q)^n \leq C \sigma(B_0)$}. 
	We obtain
	\begin{align*}
	\sum_{R \in \Top} \Vert \nabla_t \wt f \Vert_{H^1(\pom_R)} &\lesssim \Lip(f)\sum_{R\in\Top} \HH^n(2 B_0\cap\pom_R) \\ 
	&\lesssim \Lip(f) \sum_{R\in\Top}\sum_{\hat Q \cap H \neq \varnothing, \hat Q\cap 2 B_0 \neq \varnothing} \HH^n(\hat Q \cap \pom_R) \lesssim \sigma(B_0) \Lip(f).
	\end{align*}
	In the last line, we have also used the $n$-Ahlfors regularity of the boundary of the Lipschitz domains $\pom_R$.

	Summing up, we have
	\begin{align*}
		\sum_{R\in\Top,\,\, \Omega_R \cap  2B_0 \neq \varnothing} \left|\int_{\Omega_R}(\nabla w, \nabla v)\, dm\right| &\lesssim
		\sum_{R\in\Top,\,\, \Omega_R \cap 2B_0 \neq \varnothing}\Vert \nabla_t \wt f\Vert_{H^1(\pom_R)} \Vert w \Vert_{BMO(\pom_R)} \\
		&\lesssim \Vert w \Vert_{L^\infty(\pom)} \Lip(f) \cdot \sigma(B_0)
	\end{align*}
	where we have utilized that the $L^\infty$ norm is larger than the $BMO$ norm and the $L^\infty$ maximum principle for harmonic functions.
	\vvv
	
	Finally, we can control the part of the sum with the Whitney cubes intersecting $H$ as follows:
	\begin{align*}
		\sum_{\substack{\hat Q\in\mathcal W,\\\hat Q\cap H\cap 2B_0 \neq \varnothing}} \int_{\hat Q\cap H} \left|(\nabla w, \nabla v)\right| \, dm &\leq \sum_{\hat Q \cap H\cap2B_0 \neq \varnothing} m(\hat Q) \left( \fint_{\hat Q} |\nabla w|^2 \, dm \right)^{1/2} \left( \fint_{\hat Q\cap H} |\nabla \wt f|^2 \, dm \right)^{1/2} \\
		&\lesssim \operatorname{Lip}(f) \sum_{\hat Q\cap H\cap 2B_0 \neq \varnothing} m(\hat Q)\ell(\hat Q)^{-1} \left( \fint_{\hat Q} |w|^2 \, dm \right)^{1/2} \\
		&\leq\operatorname{Lip}(f) \Vert w \Vert_{L^\infty(\Omega)} \sum_{\hat Q\cap H\cap 2B_0 \neq \varnothing} m(\hat Q)\ell(\hat Q)^{-1} \\
		&\lesssim \Vert w \Vert_{L^\infty(\pom)} \Lip(f) \cdot \sigma(B_0)
	\end{align*}
	where we have used Caccioppoli's inequality and the Carleson packing condition for the Whitney cubes $\hat Q$ that intersect $H\cap2B_0$.
\end{proof}

\subsection{Existence of the weak normal derivative as a Radon measure}
{Recall that we  define the weak normal derivative of  a harmonic function $u_f$ in $\Omega$ as the functional $\partial_\nu u_f \in \Lip(\pom)^*$ satisfying}
	\[
	\int_{\Omega} \nabla u_f \nabla \wt \phi \, dm = \partial_\nu u_f(\phi),  
	\]
	{for all $\phi$ Lipschitz with compact support in $\pom$ and any Lipschitz extension $\wt \phi$}.
We aim to show a one-sided Rellich type estimate
{of the form}
\[
\Vert \partial_\nu u\Vert_{C_c(\pom)^*} \lesssim \Lip(f) \cdot \sigma(B_0)
\]
for $u$ solution of the Dirichlet problem with boundary data $f$ Lipschitz supported in $B_0\cap\pom$.

Note that the dual of the space of compactly supported continuous functions is precisely, the space of Radon measures on $\pom$ with the total variation norm.
\begin{proposition}
	\label{prop:measurebounds_weaknormalderivative}
	Let $\Omega \subset \R^{n+1}$ be a bounded corkscrew domain with uniformly $n$-rectifiable boundary, $B_0$ be a ball centered on $\pom$, and $f$ Lipschitz with support in $B_0 \cap \pom$. Then,
	the weak normal derivative of $u$ can be identified with a Radon measure and satisfies
	\[
	\Vert \partial_\nu u\Vert_{C_c(\pom)^*} \lesssim \Lip(f) \cdot \sigma(B_0).
	\]
\end{proposition}
\begin{proof}
	Given $\phi$ Lipschitz with compact support on $\pom$, let $w$ be the Dirichlet solution with boundary data $\phi$.
	Then,
	\[
	\int_{\Omega} (\nabla u, \nabla (\phi - w)) \, dm = 0
	\]
	as $\phi-w$ is in $W^{1,2}_0(\Omega)$.
	Let $v$ be the almost harmonic extension of $f$. Then, by the previous lemma, we have
	\[
	\left |\int_\Omega \nabla u \nabla \phi \,dm \right| =
	\left |\int_\Omega \nabla u \nabla w \,dm \right|
	= \left |\int_\Omega \nabla v \nabla w \, dm  \right|\lesssim \Vert \phi \Vert_{L^\infty(\pom)} \Lip(f)\cdot \sigma(B_0).
	\]
	By Riesz's representation theorem, there exists a Radon measure $\mu \in \mathcal M(\pom)$ such that
	\[
	\int_\Omega \nabla u \nabla \phi \,dm = \int_{\pom} \phi \, d\mu
	\]
	with $\Vert \mu \Vert_{\mathcal M} \lesssim \Lip(f) \sigma(B_0)$.
\end{proof}
Now, the proof of Theorem \ref{thm:URimpliesRellich} is a consequence of the atomic decomposition for the space $M^{1,1}(\pom)$.
\begin{proof}[Proof of Theorem \ref{thm:URimpliesRellich}]
	Let $f\in \Lip(\pom)\cap M^{1,1}(\pom)$. Then, by Theorem \ref{thm:atomic_decomposition}, we can write 
	\[
	f = \sum_i \lambda_i f_i
	\]
	with {$f_i$ Lipschitz atoms} and $\sum |\lambda_i| \approx \Vert \nabla_H f\Vert_{L^1}$.
	Applying Proposition \ref{prop:measurebounds_weaknormalderivative} to each $f_i$, we obtain
	\[
	\Vert \partial_\nu u_f\Vert_{C_c(\pom)^*} \lesssim \Vert \nabla_H f \Vert_{L^1(\sigma)}.
	\]
\end{proof}
Finally, the proof of Theorem \ref{thm:URimpliesweakL1} results from the boundedness of layer potentials on uniform $n$-rectifiable boundaries.
\begin{proof}[Proof of Theorem \ref{thm:URimpliesweakL1}]
Given $f\in\Lip(\pom)\cap M^{1,1}(\pom)$ , denote by $u$ the solution of the Dirichlet problem in $\Omega$ with boundary data $f$. Then, using the lemmas in Section \ref{section:layer_potentials} and Theorem \ref{thm:URimpliesRellich}, we obtain
\begin{align*}
	\Vert N(\nabla u) \Vert_{\weakL(\sigma)} &\lesssim \Vert N(\nabla D(u|_{\pom})) \Vert_{L^{1,\infty}(\sigma)} + \Vert N(\nabla S(\partial_\nu u|_{\pom})) \Vert_{\weakL(\sigma)} 
	\lesssim \Vert f \Vert_{\dot M^{1,1}(\sigma)}.
\end{align*}
\end{proof}

\subsection{Solvability in $L^p$ for $p<1$}
Next, we discuss the solvability of the regularity problem for $p<1$.
We begin by proving a localization theorem ($p<1$) for harmonic functions similar to  \cite[Theorem 5.1]{GMT}, which in turn was inspired by the one from \cite{KP1} for Lipschitz domains.

\begin{theorem}
	\label{thm:localization}
	Let $\Omega \subset \R^{n+1}$ be a corkscrew domain with uniformly $n$-rectifiable boundary $\pom$, let $x_0\in\pom$, $0<R<\frac 1 2\diam(\pom)$, and $f\in \Lip(\pom)\cap M^{1,1}(\pom)$ vanishing on $B(x_0,2R) \cap \pom$. Then, for $0<p<1$, we have
	\[
	\fint_{B(x_0, R/2)\cap\pom} \wt N_{R/2}(\nabla u)^p\, d\sigma \lesssim \left (\fint_{A(x_0,R,2R)}{|\nabla u|\, dm}\right)^p
	\]
	where $u$ is the solution to the Dirichlet problem with boundary data $f$ and $A(x_0,R,2R) = \Omega \cap B(x_0,2R) \backslash B(x_0,R)$.
\end{theorem}

\begin{proof}
	Let $\xi\in B(x_0, R/2)\cap\pom$, $x\in\gamma(\xi)\cap B(\xi, R/2)$,  $r = \dist(x,\pom)/2$, and $A_B = \overline{B(x_0, 1.8R)}\backslash B(x_0, 3/2R) $.
	
	{Following the proof of the (first part) of the localization theorem in \cite[Proposition 5.1]{GMT}}, we obtain
	\[
	\left( \fint_{B_{r/2}(x)} |\nabla u(z)|^2 \, dz \right)^{1/2}  \lesssim  \frac 1 {Rr} \sup_{t \in 1.05 A_B} G(x,t) \int_{1.05 A_B} |\nabla u| \,dy.
	\]
	Since the Green function $G(x,\cdot)$ is positive and subharmonic in $1.05 A_B$, by the {de Giorgi's boundedness theorem (see \cite[Theorem 4.1]{HL}  for example)}, we have
	\[
	\sup_{t \in 1.05 A_B} G(x,t) \lesssim_p \left(\fint_{1.1A_B} G(x,y)^p \, dy \right)^{1/p} \lesssim \left(\fint_{1.1 A_B} \left ( \frac{\omega^y(\Delta_\xi)}{\delta_\Omega(x)^{n-1}} \right)^p \, dy \right)^{1/p}
	\]
	for any $p>0$.
	
	Furthermore, using the bound $\omega^y(\Delta_\xi) \delta_\Omega(x)^{-n} \lesssim M_{c,\sigma}(\omega^y)(\xi)$ where $M_c$ is the centered Hardy-Littlewood maximal operator,
	we get
	\begin{align*}
		\label{eq:localization_thm_important_bound}
		\fint_{B(x_0,R/2)\cap\partial\Omega} &|\wt N_{R/2}(\nabla u)(\xi)|^p d\sigma(\xi) \nonumber\\
		&\hspace{-5mm}\lesssim \fint_{B(x_0,R/2)} \fint_{1.1 A_B} M_{c,\sigma} \omega^y|_{B(x_0,CR)}(\xi)^p dy d\sigma(\xi) \cdot 
		\left(\sigma(B(x_0,R))\fint_{1.1A_B} |\nabla u|\,dy \right)^p \\
		&\hspace{-5mm}= \fint_{1.1A_B} \fint_{B(x_0, R/2)} M_{c,\sigma} \omega^y|_{B(x_0,CR)}(\xi)^p d\sigma(\xi) dy \left ( \sigma(B(x_0,R))\fint_{1.1A_B} |\nabla u|\,dy \right)^p \nonumber\\
		&\hspace{-5mm}\lesssim \fint_{1.1A_B} dy \cdot  \sigma(B(x_0, R/2))^{-p} \left ( \sigma(B(x_0,R)) \fint_{1.1 A_B} |\nabla u |\,dy \right)^{p}\\
		&\hspace{-5mm}\lesssim \left ( \fint_{1.1A_B} |\nabla u| \,dy \right)^p \nonumber
	\end{align*}
	where we used Fubini (in the second to third line),
	{Kolmogorov's inequality (see for example \cite[Lemma 5.16]{Duo})} and that $M_c$ is bounded from the space of Radon measures to $L^{1,\infty}$. 
	
\end{proof}

Now, the proof of {Corollary} \ref{coro:solvabilityRp} follows the same steps as the proof of the extrapolation of the solvability of $(R_p)$ for $p>1$ presented in \cite{GMT}.
\begin{proof}[Proof of Corollary \ref{coro:solvabilityRp}]
The proof is analogous to the proof of \cite[Theorem 1.3]{GMT} by substituting H\"older's inequality for {Kolmogorov's inequality}
and the initial assumption of $(R_p^\Delta)$ solvability for $(R_{1,\infty}^\Delta)$ solvability.

	
%
%
%
%
%

\end{proof}
\vv

\section{Counterexamples on the $4$-corners Cantor set}
\label{section:Cantor}



\subsection{Definition of $4$-corners Cantor set}
We construct the $4$-corners Cantor set with the following algorithm. Consider the unit square \( \hat Q_0 = [0, 1] \times [0, 1] \). At the first step, we take 4 closed squares inside \( \hat Q_0 \), of side-length \( 1/4 \), with sides parallel to the coordinate axes, such that each square contains a vertex of \( \hat Q_0 \). At step 2, we apply the preceding procedure to each of the 4 squares produced at step 1. Then we obtain 16 squares of side-length \( 1/16 \). Proceeding inductively, we have at the \( n \)-th step $4^n$ squares \( \hat Q_n^j \), \( 1 \leq j \leq 4^n \), of side-length \( 1/4^n \). 

Write
\[
E_n = \bigcup_{j=1}^{4^n} \hat Q_n^j,
\]
and we define the $4$-corners Cantor set as 
\[
E = \bigcap_{n=1}^{\infty} E_n.
\]
We can consider the natural dyadic structure $\Dsigma$ on $E$ given by $\Dsigma^j = \{\hat Q_j^i\cap E\}_{i=1}^{4^j}$.

Let $\mathcal L = \operatorname{div} A\nabla$ be an elliptic operator with the matrix $A$ satisfying \eqref{eq:ellipticity_conditions_operator}, and $\omega$ its associated elliptic measure.
Let $\Omega = B(0,100) \backslash E$, $\sigma := \HH^1|_\pom$, and $p = (20,0) \in \Omega$ be a pole far away from $\pom$, and $\Dsigma$ be the dyadic structure previously defined on $E$. 
For $Q\in \Dsigma$, we define $f_Q(x) := \ell(Q)\chi_Q(x)$ and let $u_Q$  be the solution of the Dirichlet problem on $\Omega$ with boundary data $f_Q$. Also, to every cube $Q$, we denote by $x_Q$ the center of mass of the associated cube $\hat Q_n^j$ inside $\mathbb R^2$.

\subsection{Solutions of elliptic PDEs on the complementary of the $4$-corners Cantor set}
In what follows, we will show that for all $C>0$, there exists a $1$-Lipschitz functions $f$ on $\pom$ satisfying
\begin{align*}
	\Vert \wt N(\nabla u_f)\Vert_{L^{1,\infty}(\sigma)} &\geq C 
\end{align*}
where $u_f$ is the solution of the Dirichlet problem with boundary data $f$ for the operator $\mathcal L$.
To do so, we will find a family of uniformly Lipschitz functions such that, if we choose a function at random, we expect that this inequality will  hold. 

We will start by introducing our family of functions and showing that these are uniformly Lipschitz.
{We could also use the functions we will define in the following  section as $\pom$ does not satisfy WNB but we have decided to include a (much easier) construction.}
\begin{lemma}
	\label{lemma:Lipschitz_functions_cantor}
	For any choice of signs $(\epsilon_Q)_{Q\in\Dsigma}$ (each $\epsilon_Q \in \{-1,+1\}$), the function 
	$f = \sum_{Q\in\Dsigma} \epsilon_Q f_Q $
	is Lipschitz with bounded constant on $\pom$.
\end{lemma}
\begin{proof}
	First, we will extend the functions $f_Q$ to a Lipschitz function in $\mathbb R^2$.
	Let $x_Q$ be the center in $\mathbb R^2$ of the cube $Q$ and consider the closed square $\hat Q\subset\mathbb R^2$ with center $x_Q$ and sidelength $1.05\ell(Q)$.
	Then, let $\hat f_Q$ be a $C^1$ bump function that is $\ell(Q)$ in $\hat Q$ and $0$ outside of $1.1 \hat Q$ and with $| \nabla \hat f_Q|\lesssim 1$. Note that $\supp \nabla \hat f_Q \subset 1.1\hat Q \backslash \hat Q$, and since the sets $1.1 \hat Q \backslash \hat Q$ have finite overlap for $Q\in\Dsigma$, we have
	\[
	\left\Vert \sum_Q \epsilon_Q |\nabla\hat f_Q| \right\Vert_{L^\infty(\mathbb R^2)} \lesssim 1
	\]
	for any choice of signs $(\epsilon_Q)_Q$. Hence, the function $\sum_{Q} \epsilon_Q \hat f_Q$ is Lipschitz on $\mathbb R^2$ and its restriction to $E$ coincides with $f =\sum_Q\epsilon_Q f_Q$.
\end{proof}

%
%

Next, we prove a lemma about the gradient of the functions $u_Q$.
\begin{lemma}
	\label{lemma:behavior_uQ_cantor}
	There exists $0<c<\sqrt{2}/4$ such that for all $R,Q\in\Dsigma$ with $R\subseteq Q$, we have
	\[
	\fint_{\hat B_R}|\nabla u_Q|\, dm \approx
	\left(\fint_{\hat B_R}|\nabla u_Q|^2\, dm\right)^{1/2}
	\approx \frac{\ell(Q) \omega^p(R)}{\ell(R)\omega^p(Q)}
	\]
	where $\hat B_R = B(x_R, c\ell(R))$ and the comparability constants depend on $c$ and $\mathcal L$ but not on $Q$ or $R$. 
\end{lemma}
\begin{proof}
	We have
	\[
	u_Q(x) = \ell(Q) \omega^{x}(Q)\quad \mbox{and} \quad\ell(Q)-u_Q(x) = \ell(Q)\omega^{x}(\pom\backslash Q). 
	\]
	Since $\omega^{x}(\pom\backslash Q)$ is a positive solution vanishing on $Q$, by {boundary Harnack inequality \ref{thm:bdryHarnack} and  Lemma \ref{lemma:Bourgain}}, we have
	\[
	\omega^{x}(\pom\backslash Q) \approx \frac{G^p(x)}{G^p(x_Q)}, \quad \mbox{for }x\in\hat B_R
	\]
	and, by Theorem \ref{thm:comparabilityomegaGreen}, we obtain
	\[
	\frac{G^p(x)}{G^p(x_Q)} \approx \frac{\omega^p(R)}{\omega^p(Q)}, \quad \mbox{for $x\in\hat B_R$}
	\]
	with comparability constant depending on $c$.
	Using Caccioppoli, we can show 
	\begin{align*}
	\fint_{\hat B_R} |\nabla u_Q|^2 \, dm = \fint_{\hat B_R}|\nabla(\ell(Q)-u_Q)|^2\, dm  &\lesssim \frac{1}{\ell(R)} \fint_{(1+\epsilon)\hat B_R}|\ell(Q)-u_Q|^2\, dm \\ &
	\lesssim \left(\frac{\ell(Q)}{\ell(R)} \frac{\omega^p(R)}{\omega^p(Q)}\right)^2
	\end{align*}
	again with constants depending on $c$.

	On the other hand, if we choose $c$ close enough to $\sqrt 2/4$, we can find a small ball $b\subset \hat B_R$ close to $R$ with $r(b)\approx \ell(R)$ such that 
	\[
	\omega^x(\pom\backslash Q) \leq 0.1\, \omega^{x_R}(\pom\backslash Q), \quad \mbox{for $x\in b$}
	\]
	using Lemma \ref{lemma:Holderdecay} (H\"older continuity for solutions). Now by Poincar\'e's and Harnack's inequalities, 
	\begin{align*}
		\ell(R) \fint_{\hat B_R} |\nabla u_Q| \, dm &\gtrsim
		\fint_{\hat B_R} \left|\ell(Q)- u -\left(\ell(Q)- \fint_{\hat B_R} u\, dm\right)\right|\, dm\\ 
		&\gtrsim \ell(Q)
		\fint_{b \cup B(x_R, r(b))} \left| \omega^x(\pom\backslash Q) -\left( \fint_{b \cup B(x_R, r(b))} \omega^x(\pom\backslash Q)\, dm\right)\right|\, dm\\ 
		&\gtrsim \ell(Q)\, \omega^{x_R}(\pom\backslash Q)
	\end{align*}
	which finishes the proof.
\end{proof}

\subsection{Counterexample to $(R_{1,\infty}^{\mathcal L})$}

In order to show the counterexample to the solvability of $(R_{1,\infty}^{\mathcal L})$, we will start by introducing a linearization of the modified nontangential maximal operator $\wt N$. Fix $k\in\mathbb N$, and
let $(Q^k_j)_j \subset \Dsigma$ be the set of maximal cubes $Q\in\Dsigma$ satisfying
\[
\frac{\omega^p(Q)}{\ell(Q)}\frac{\ell(P^i) }{\omega^p(P^i)} \geq 1, \quad 1\leq i\leq k-1
\]
where $P^i$ is the $i$-th ancestor of $Q$.
We define the operator $\hat N_k$ for functions $F\in L^1_{\operatorname{loc}}(\Omega)$ as follows
\[
\hat N_k F(\xi) := \begin{cases}
	\fint_{\hat B_{Q^k_j}}|F|\, dm, & \mbox{if $\xi\in Q^k_j$ for some $j$},\\
	0, &\mbox{ if $\xi\in\pom\backslash \cup_j Q^k_j$}
\end{cases}
\]
{where $\hat B_{Q^k_j}$ are the balls given by Lemma \ref{lemma:behavior_uQ_cantor}.}
Note that for any function $F$, we trivially have 
\[
\wt N(F) \geq \hat N_k(F)
\] 
assuming that the nontangential cones in the definition of $\wt N$ are wide enough.
\begin{lemma}
	The following statements hold:
	\begin{itemize}
		\item The maximal cubes satisfy $\sigma(E \backslash \cup_j Q^k_j) = 0$, for any $k$,
		
		\item  $ \wt N( \mathbb E|\nabla u|)(\xi) \geq  \mathbb E  \hat N_k(|\nabla u|)(\xi) 
		\geq \sqrt{k}$ for $\xi \in Q^k_j$.	
	\end{itemize}
\end{lemma}
\begin{proof}
	We will show that 
	\[
	\frac{\sigma(R\cap \cup_j Q^k_j)}{\sigma(R)} > c(k)>0, \quad\mbox{ for all $R\in \Dsigma$,}
	\]
	which implies the first statement.
	Fix $R\in\Dsigma$. By pigeonhole principle, one of its children $R_1$ belongs to $\cup_j Q^1_j$. In turn, one of the children $R_2$ of $R_1$, belongs to $\cup_j Q^2_j$. Iterating, there is at least one cube $k$ generations lower than $R$ inside $\cup_j Q^k_j$.
	
	\vv
	For the second statement, if the cones $\gamma$ in the definition of $\wt N$ are wide enough, then $\hat B_R \subset \gamma(\xi)$ for all $\xi \in R$, and the first inequality follows. For the second one, assume $\xi \in Q^k_j$. We have 
	\begin{align*}
		\mathbb E \hat N_k(|\nabla u|)(\xi) &= 
		\fint_{\hat B_{Q^k_j}} \mathbb E |\nabla u|\, dm \approx
		\fint_{\hat B_{Q^k_j}} \left(\sum_Q |\nabla u_Q|^2\right)^{1/2}\, dm   \\
		&\geq \fint_{\hat B_{Q^k_j}} \Bigg(\sum_{Q: Q^k_j\subset Q,\, \frac{\omega^p(Q^k_j)}{\ell(Q^k_j)}\frac{\ell(Q) }{\omega^p(Q)} \geq 1} |\nabla u_Q|^2\Bigg)^{1/2}\, dm  \gtrsim \sqrt k
	\end{align*}
	where we have used {Khintchine's inequality} \ref{prop:Khintchine}, the properties of $Q^k_j$, Cauchy-Schwarz, and Lemma \ref{lemma:behavior_uQ_cantor}.
\end{proof}

%

The existence of the desired counterexample will follow from the next result.
\begin{proposition}
	\label{prop:expectation_N_weakL1}
	We have
	\[
	\mathbb E \,\Vert \hat N_k(|\nabla u|)\Vert_{L^{1,\infty}} 
	\gtrsim \sqrt k \sigma(E)
	\]
	with constants independent of $k$.
\end{proposition}

\begin{proof}[Proof of Proposition \ref{prop:expectation_N_weakL1}]
	By linearity of expectation, we have
	\[
	\mathbb E \, \sigma \left( \left \{\xi\in E: \hat N_k(|\nabla u|)(\xi) > \lambda \right\}\right) = \sum_{Q^k_j} \mathbb E \, \sigma \left( \left\{ \xi\in Q^k_j : \hat N_k(|\nabla u| )(\xi) > \lambda \right\}\right), \quad \mbox{for all $\lambda>0$}.
	\]
	Since the function $\hat N_k(|\nabla u|)$ is constant on every $Q^k_j$, we can fix $\xi^k_{j}\in Q^k_j$ and by the Paley-Zygmund inequality we have
	\begin{align*}
		\mathbb E \, \sigma \big( \big\{ \xi \in Q^k_j :  \hat N_k(|\nabla u |)(\xi) > \frac 1 2 \mathbb E &[\hat N_k(|\nabla u|)(\xi)] \big\} \big) \\
		&= \sigma(Q^k_j) \mathbb P  \left(   \hat N_k(|\nabla u |)(\xi^k_j) > \frac 1 2 \mathbb E[\hat N_k(|\nabla u|)(\xi^k_j)]  \right)\\
		&\gtrsim \sigma(Q^k_j).
	\end{align*}
	Hence,
	\begin{align*}
		\mathbb E \,\Vert \hat N_k(|\nabla u|)\Vert_{L^{1,\infty}} 
		&\geq \frac{\sqrt k}2 \sum_{Q^k_j} \mathbb E \, \sigma \left( \left\{ \xi\in Q^k_j : 
		\hat N_k(|\nabla u| )(\xi) > \frac {\sqrt k}2 \right\}\right)	\\
		&\geq 
		\frac{\sqrt k}2 \sum_{Q^k_j} \mathbb E \, \sigma \left(\left\{  \xi\in Q^k_j : \hat N_k(|\nabla u|) > \frac 1 2 \mathbb E[\hat N_k( |\nabla u|)]\right\} \right)
		\\
		&\gtrsim \sqrt{k}  \sum_{Q^k_j} \sigma(Q^k_j) \approx \sqrt k \sigma(E) 
	\end{align*}
	where we have used that $\mathbb E\hat N_k(|\nabla u|) \geq \sqrt k$ on $Q^k_j$.

\end{proof}
Since this result is true for every $k$, a uniform bound of the form 
\[
\Vert \wt N(\nabla u_f)\Vert_{L^{1,\infty}} \lesssim \Vert \nabla_H f \Vert_{L^1} 
\]
is not possible, yielding the proof of Proposition \ref{prop:NweakL1inCantor}.

\subsection{$(D_{p'}^{\mathcal L^*})$ does not imply $(R_p^{\mathcal L})$ for general Ahlfors regular domains}
{In this section, we show that for general operators the Dirichlet problem and the regularity problem are not equivalent.}

\begin{theorem}
	\label{thm:extrapolationregularity}
	Let $\Omega \subset \mathbb R^{n+1}$ be a corkscrew domain with $n$-Ahlfors regular boundary $\pom$. Then, solvability of  $(R_p^{\mathcal L})$ implies solvability of $(R_{1,\infty}^{\mathcal L})$.
\end{theorem}
This is the statement of \cite[Theorem 1.3]{GMT}, noting the trivial inequality 
\[
\Vert \wt N(\nabla u) \Vert_{L^{1,\infty}} \leq \Vert \wt N(\nabla u) \Vert_{L^{1}}. 
\]
\begin{proposition}
	Let $E\subset \mathbb R^2$ be the four corners Cantor set and $\Omega = \mathbb R^2\backslash E$. Let $\mathcal L^*$ be an operator on $\Omega$ such that its elliptic measure $\omega_{\mathcal L^*}$ is comparable to surface measure on $E$. Then, $(R_p^{\mathcal L})$ is not solvable for any $p>1$.
\end{proposition}
\begin{remark}
	Note that such an operator $\mathcal L^*$ of the form $\operatorname{div}(a\nabla \cdot)$ with $C^{-1}\leq a(x)\leq C$ exists as shown by David and Mayboroda in \cite{DM}.
\end{remark}
\begin{proof}
	Using the characterization of solvability of the Dirichlet problem by reverse H\"older inequalities for the density of the elliptic measure, we have that $(D^{\mathcal L^*}_{p'})$ is solvable for all $p'>1$. 
	On the other hand, by Proposition \ref{prop:NweakL1inCantor} we know that $(R_{1,\infty}^\mathcal L)$ is not solvable. Hence, by Theorem \ref{thm:extrapolationregularity}, $(R_{p}^\mathcal L)$ is not solvable for any $p\geq1$.
\end{proof}

\section{Domains not satisfying WNB}
\label{section:WNB}
In this section, we will show that the validity of the one-sided Rellich inequality on {a corkscrew domain $\Omega \subset \R^{n+1}$ with $n$-Ahlfors regular boundary} implies the WNB for $\pom$. To do so, we will assume $\pom$ does not satisfy the WNB and for every $C>0$ {we will construct a family of Lipschitz functions such that if we choose one at random, the solution of the Dirichlet problem is expected to fail \eqref{eq:rellich_ineq} with constant $C$.}
\subsection{Preliminaries}
Given $\epsilon>0$,
\begin{itemize}
 \item the \textit{standard $\epsilon$-box} is the subset $M(\epsilon) \subset \mathbb R^{n+1}$  given by $M(\epsilon) = B_i \cup B_m \cup B_o \cup S_i \cup S_m \cup S_o$, where 
\item $B_i = \{(x_1,\hdots, x_{n+1})\in\mathbb R^{n+1} \,:\, 2\epsilon/3\leq x_{n+1}\leq \epsilon, \, \sup_{i=1,\hdots,n} |x_i|\leq 1-2\epsilon/3\}$ is the \textit{inner base of the box}, 
\item $B_m= \{(x_1,\hdots, x_{n+1})\in\mathbb R^{n+1} \,:\, \epsilon/3\leq x_{n+1}\leq 2\epsilon/3, \, \sup_{i=1,\hdots,n} |x_i|\leq 1-\epsilon/3\}$ is the \textit{middle base of the box}, 
\item $B_o = \{(x_1,\hdots, x_{n+1})\in\mathbb R^{n+1} \,:\, 0\leq x_{n+1}\leq \epsilon/3, \, \sup_{i=1,\hdots,n} |x_i|\leq 1\}$ is the \textit{outer base of the box}, 
\item $S_i = \{(x_1,\hdots, x_{n+1})\in\mathbb R^{n+1} \,:\, \epsilon\leq x_{n+1}\leq 10\epsilon, \, 1-\epsilon\leq\sup_{i=1,\hdots,n} |x_i|\leq 1-2\epsilon/3\}$ is the \textit{inner side of the box},  
\item $S_m = \{(x_1,\hdots, x_{n+1})\in\mathbb R^{n+1} \,:\, 2\epsilon/3\leq x_{n+1}\leq 10\epsilon, \, 1-2\epsilon/3\leq\sup_{i=1,\hdots,n} |x_i|\leq 1-\epsilon/3\}$ is the \textit{middle side of the box}, 
\item $S_o = \{(x_1,\hdots, x_{n+1})\in\mathbb R^{n+1} \,:\, \epsilon/3\leq x_{n+1}\leq 10\epsilon, \, 1-\epsilon/3\leq\sup_{i=1,\hdots,n} |x_i|\leq 1\}$ is the \textit{outer side of the box}.
\item We call the set $C = \{(x_1,\hdots, x_{n+1})\in\mathbb R^{n+1} \,:\, \epsilon\leq x_{n+1}\leq 2\epsilon, \, \sup_{i=1,\hdots,n} |x_i|\leq 1-\epsilon\}$ the \textit{content of the box $M(\epsilon)$}. 
\item We also denote by $T = \{(x_1,\hdots, x_{n+1})\in\mathbb R^{n+1} \,:\,  x_{n+1}= 10\epsilon, \, \sup_{i=1,\hdots,n} |x_i|\leq 1\}$ the \textit{top of the box}, 
\item $I = \{(x_1,\hdots, x_{n+1})\in\mathbb R^{n+1} \,:\,  \epsilon <x_{n+1}<9\epsilon, \, \sup_{i=1,\hdots,n} |x_i|\leq 1-\epsilon\}$ the \textit{interior of the box}.
\end{itemize}

We finally define $\pi_{n+1}$ as the projection onto the $(n+1)$-th axis. {See Figure \ref{figure:box} for a graphical representation of the standard $\epsilon$-box in the plane.}

\begin{figure}[h]

	\centering
\begin{tikzpicture}[scale=5]
	\centering
	\def\eps{1/7}
	
	\pgfmathsetmacro{\epsthird}{\eps/3}
	\pgfmathsetmacro{\epstwoThird}{2*\eps/3}
	\pgfmathsetmacro{\oneMinusEpsThird}{1-\eps/3}
	\pgfmathsetmacro{\oneMinusTwoEpsThird}{1-2*\eps/3}
	\pgfmathsetmacro{\oneMinusEps}{1-\eps}
	\pgfmathsetmacro{\tenEps}{10*\eps}
	\pgfmathsetmacro{\twoEps}{2*\eps}
	
	
	\foreach \y in {0,\epsthird,\epstwoThird,\eps,\twoEps, 9*\eps, \tenEps} {
		\draw[dashed] (-1.1,\y) -- (1.1,\y);
	}
	
	\node[left] at (-1.1,\eps) {$\epsilon$};
	\node[left] at (-1.1,2*\eps) {$2\epsilon$};
	\node[left] at (-1.1,9*\eps) {$9\epsilon$};
	\node[left] at (-1.1,10*\eps) {$10\epsilon$};
	
	\fill[blue!20] (-1,0) rectangle (1,\epsthird);
	
	\fill[blue!40] (-\oneMinusEpsThird,\epsthird) rectangle (\oneMinusEpsThird,\epstwoThird);
	
	\fill[blue!60] (-\oneMinusTwoEpsThird,\epstwoThird) rectangle (\oneMinusTwoEpsThird,\eps);
	
	\fill[red!20] (-1,\epsthird) -- (-\oneMinusEpsThird,\epsthird) -- (-\oneMinusEpsThird,\tenEps) -- (-1,\tenEps) -- cycle;
	\fill[red!20] (1,\epsthird) -- (\oneMinusEpsThird,\epsthird) -- (\oneMinusEpsThird,\tenEps) -- (1,\tenEps) -- cycle;
	
	\fill[red!40] (-\oneMinusEpsThird,\epstwoThird) -- (-\oneMinusTwoEpsThird,\epstwoThird) -- (-\oneMinusTwoEpsThird,\tenEps) -- (-\oneMinusEpsThird,\tenEps) -- cycle;
	\fill[red!40] (\oneMinusEpsThird,\epstwoThird) -- (\oneMinusTwoEpsThird,\epstwoThird) -- (\oneMinusTwoEpsThird,\tenEps) -- (\oneMinusEpsThird,\tenEps) -- cycle;
	
	\fill[red!60] (-\oneMinusTwoEpsThird,\eps) -- (-\oneMinusEps,\eps) -- (-\oneMinusEps,\tenEps) -- (-\oneMinusTwoEpsThird,\tenEps) -- cycle;
	\fill[red!60] (\oneMinusTwoEpsThird,\eps) -- (\oneMinusEps,\eps) -- (\oneMinusEps,\tenEps) -- (\oneMinusTwoEpsThird,\tenEps) -- cycle;
	
	\fill[green!40, opacity=0.6] (-\oneMinusEps,\eps) rectangle (\oneMinusEps,\twoEps);
	\fill[green!20, opacity=0.6] (-\oneMinusEps,2*\eps) rectangle (\oneMinusEps,9*\eps);
	
	\draw[thick] (-1,\tenEps) -- (1,\tenEps);

	\begin{scope}[shift={(1.2,1.0)}, scale=0.4]
		\fill[blue!20] (0,0) rectangle (0.2,0.2);
		\node[right] at (0.25,0.1) {Outer Base};
		
		\fill[blue!40] (0,-0.3) rectangle (0.2,-0.1);
		\node[right] at (0.25,-0.2) {Middle Base};
		
		\fill[blue!60] (0,-0.6) rectangle (0.2,-0.4);
		\node[right] at (0.25,-0.5) {Inner Base};
		
		\fill[red!20] (0,-0.9) rectangle (0.2,-0.7);
		\node[right] at (0.25,-0.8) {Outer Side};
		
		\fill[red!40] (0,-1.2) rectangle (0.2,-1.0);
		\node[right] at (0.25,-1.1) {Middle Side};
		
		\fill[red!60] (0,-1.5) rectangle (0.2,-1.3);
		\node[right] at (0.25,-1.4) {Inner Side};
		
		\fill[green!40] (0,-1.8) rectangle (0.2,-1.6);
		\node[right] at (0.25,-1.7) {Content};
		\fill[green!20] (0,-2.1) rectangle (0.2,-1.9);
		\node[right] at (0.25,-2) {Interior};
	\end{scope}
	
\end{tikzpicture}
\caption{{Graphical representation of the standard box $M(\epsilon)$ in the plane with $\epsilon = 1/7$.} {The box $M$ has width $2$ and height $10\epsilon$.}}
	\label{figure:box}
\end{figure}
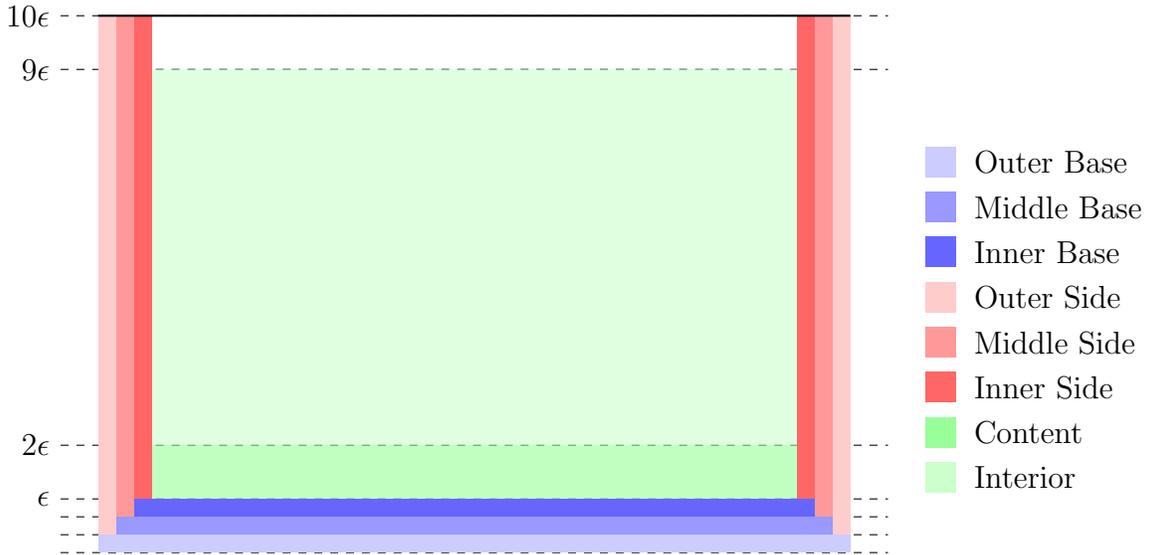


Let $E$ be an $n$-Ahlfors regular set in $\mathbb R^{n+1}$. For each $\epsilon>0$, let $\mathcal G(\epsilon)$ denote the set of points $(x,t)\in E \times \mathbb R_+$ for which it is impossible to find a translation and dilation of the standard $\epsilon$-box $M$ contained in $B(x,t)$ such that
\[
\begin{cases}
	\diam M \geq \epsilon t, \\
	M\cap E = \varnothing, \mbox{ but $E$ intersects the content of $M$}.
\end{cases}
\]
\begin{definition}We say that a set $E$ satisfies the WNB (weak-no-box condition) if for any rotation $\rho$ of the set, $\rho(E) \times \mathbb R_+ \backslash \mathcal G(\epsilon)$ is a Carleson set.
\end{definition}
In \cite[pages 280 to 283]{DS2}, the authors show another characterization of the WNB which will be useful in the sequel.
\begin{definition}
	\label{def:bad_boxes}
	Given $\mu,\epsilon>0$, we denote by $\mathcal B(\epsilon,\mu)$ be the set of $(x,t)\in  E\times \mathbb R_+ \backslash \mathcal G(\epsilon)$ such that there exists a {dilation and translation $M$ of the standard $\epsilon$-box with $M\subset B(x,t)$, $M\cap E = \varnothing$, $\diam M \geq \epsilon t$} and points $x^j \in E,\, 1\leq j \leq n+2$ which lie inside the interior of the box $M$ and that satisfy $\pi_{n+1}(x^j) < \pi_{n+1}(x^{j+1}) - \mu \diam M$ for $j=1,\hdots,n+2$.
\end{definition}
\begin{proposition}
	\label{prop:characterizationWNB}
	 Assume $E$ $n$-Ahlfors regular does not satisfy the WNB. There exist {$\epsilon>0$ and} $\mu = \mu(\epsilon)>0$ such that, after a rotation of the set $E$, the
	 set $\mathcal B(\epsilon,\mu)$ is not Carleson. 
\end{proposition}
{The main idea behind the proof of the previous proposition is that one can slightly rotate the boxes, and make $\epsilon$ smaller to guarantee the existence of the points $x^j$. Finally, choose a finite number of rotations in $\mathbb R^n$ such that any possible rotation is very close to one of the selected. The full argument can be found in \cite[Section 4.2]{DS2}.} 

A nice property of the WNB for $1$-Ahlfors regular sets is that it characterizes uniform $1$-rectifiability.
\begin{theorem}[Proposition 3.45 and 3.50 from \cite{DS2}]
	\label{thm:WNBiffUR}
	A $1$-Ahlfors regular set $E\subset \mathbb R^{n+1}$ satisfies the WNB if and only if $E$ is uniformly $1$-rectifiable.
\end{theorem}
{Unfortunately, the WNB does not characterize uniform rectifiability for $n\geq2$ as there exist $n$-Ahlfors regular sets which are purely $n$-unrectifiable but satisfy WNB: in $\mathbb R^3$ take the product of the $4$-corners Cantor set inside $\mathbb R^2$ with $\mathbb R$ (see \cite[pg 150]{DS2} for more details on this counterexample).}

The WNB was introduced in \cite{DS2} as a geometric condition on a set $E$ 
such that its negation implies that there exist Lipschitz functions on $E$ 
that cannot be well approximated by affine functions at most scales (specifically, $E$ does not satisfy the WALA, see Section \ref{section:history} for its definition). It is known that uniform rectifiability of $E$ implies the WALA, but the converse is still an open question for $n$-Ahlfors regular sets with $n\geq2$.
The failure of the WNB is very useful to construct functions failing the WALA, as the geometry of the boxes is suitable to construct Lipschitz function with special behavior.

In \cite[Section III.4]{DS2}, David and Semmes build Lipschitz functions with the property that at the scale of each box, the function is badly approximated by any affine function. We require different properties for our functions. We want a family of functions $(f_Q)_Q$ satisfying that a sum with random signs $f= \sum_{Q} \epsilon_Q f_Q$, $\epsilon_Q \in \{\pm1\}$ is Lipschitz with uniform constant, with each $f_Q$ satisfying that $\Vert \partial_\nu u_{Q} \Vert_{\mathcal M} \gtrsim \sigma(Q)$, where $u_{Q}$ is the solution of the Dirichlet problem with boundary data $f_Q$.
{We will construct these functions in the following section, but for a simpler version on how to choose functions with similar behavior in the $4$-corners Cantor set, see Lemma \ref{lemma:Lipschitz_functions_cantor}.}

\subsection{Constructions of bad functions}
\label{section:constructionbadfunctions}
Let's start by considering a discretized version of the WNB. 
Assume $\pom$ does not satisfy WNB for some $\epsilon>0$ {in the sense of Proposition \ref{prop:characterizationWNB}, and assume without loss of generality $\mu \ll \epsilon$}. {For $Q\in\Dsigma$} such that there exists  $x\in Q$ and $t=\ell(Q)$ with $(x,t)\in\mathcal B(\epsilon, \mu)$, choose a box $M_Q\subset \mathbb R^{n+1}$ {satisfying the properties in Definition \ref{def:bad_boxes}} except if we have already selected a box $M_{Q'}$ for some $Q'$ with $\frac{\mu^2}{100n}\diam M_{Q'} \leq \diam M_Q \leq 100n\mu^{-2} \diam M_{Q'}$, and $\dist(Q,Q') \leq 10(\ell(Q)+\ell(Q'))$. We denote by $\mathcal B_d(\epsilon)$ the set of cubes $Q$ for which we have fixed such a box $M_Q$. 
\begin{remark}
	Note that if $\mathcal B(\epsilon, \mu)$ is not a Carleson family, then $\mathcal B_d(\epsilon)$ is not either. 
\end{remark}
\begin{remark}
	Let $M_1$ and $M_2$ be boxes corresponding to $Q_1,Q_2 \in\mathcal B_d(\epsilon)$ with $\diam M_1 \leq \diam M_2$. Without loss of generality assume that $M_2$ is the standard $\epsilon$-box.
	\begin{itemize}
		\item If their interiors $I_1$ and $I_2$ have non-empty intersection, then the convex hull of $M_1$ is contained inside the convex hull of the inner sides and inner base of $M_2$.
		
		\item If $M_1$ intersects the middle side of $M_2$, then $M_1 \subset \{x_{n+1}\geq 9\epsilon\}$. $M_1$ cannot intersect the middle base of $M_2$.
	\end{itemize} 
	These properties are a consequence of the fact that the boxes $M_1$ and $M_2$ do not intersect $\pom$ but  their interiors $I_1,I_2$ do, and that if $M_1$ and $M_2$ are close, then $\diam M_1 < \epsilon/3$.
	\color{black}
\end{remark}

Let $\mathcal F$ be a finite subfamily of $\mathcal B_d(\epsilon)$ and $Q\in\mathcal F$. In what follows, we will define a Lipschitz function $f_Q$ supported in the convex hull of $M_Q$ and with Lipschitz constants independent of $Q$.
Without loss of generality assume the box $M_Q$ is the standard $\epsilon$-box. 
Let 
$$\mathcal F_Q := \{R\in \mathcal F\, :\, I_R \cap I_Q \neq \varnothing,\, \diam M_R<\diam M_Q\},$$ 
that is, the set of cubes $R$ of $\mathcal F$ such that the interior $I_R$ of its associated box $M_R$ intersects the interior $I_Q$, and $\diam M_R < \diam M_Q$ (in particular $Q\not\in \mathcal F_Q$).
We enumerate the elements of $\mathcal F_Q$ as $\{Q_i\}$ so that $i<j$ implies $\diam M_{Q_i} \geq \diam M_{Q_j}$.

In the interior $I_Q$ of the box $M_Q$, define
$$
f_0(x) = (9\epsilon - x_{n+1})/\epsilon.
$$ 
Then, let $f_Q^0(x)$ be a $C\epsilon^{-1}$-Lipschitz extension of $f_0$ to $\mathbb R^{n+1}$ such that {$f_Q^0(x) = \max\{0,(9\epsilon - x_{n+1})/\epsilon\}$ for $x$ in the convex hull of $M_{Q_j}$ for $Q_j\in\mathcal F_Q$}, and $f_Q^0 \equiv 0$ outside the convex hull of the inner base and sides of $M_Q$ and for $x_{n+1}\geq 9\epsilon$.

{By induction, we will define  $C\epsilon^{-1}$-Lipschitz functions $(f_Q^j)_{j=1}^{|\mathcal F_Q|}$ satisfying that $f_Q^j$ is constant on the top of every box $M_{Q_i}$ for $Q_i\in \mathcal F_Q$, $i\geq j$}. For $j\geq1$, if the box $M_j$ corresponding to $Q_j \in \mathcal F_Q$ is contained inside the convex hull of a larger box in $\mathcal F_Q$, let $f_Q^j = f_Q^{j-1}$. {Else, if $M_j$ is not completely inside the convex hull of a larger box in $\mathcal F_Q$, let} {$c_{j-1}=f_Q^{j-1}(y_{T_j})$ for some $y_{T_j} \in T_j$, the top of the box $M_j$, and}
\[
f_Q^j(x) = \begin{cases}
	f_Q^{j-1}(x), \quad  \mbox{in the outer base and sides and outside the convex hull of $M_j$,}\\
	{c_{j-1}}\color{black}, \quad \mbox{in the convex hull of the inner base and sides of $M_j$,}\\
	\mbox{a $C\epsilon^{-1}$-Lip extension of the previous function}, \, \mbox{in the middle base and sides.}
\end{cases}
\]
The definition is independent of the choice of  $y_{T_j}$  as $f_Q^{j-1}$ is constant on $T_j$ as, by construction, $T_j$ cannot intersect the middle base or sides of any larger box in $\mathcal F_Q$ and $f_Q^0$ is constant on $T_j$.
We finally define $f_Q$ equal to $f_Q^{|\mathcal F_Q|}$ as $\mathcal F$ is finite.

\begin{remark}
We gather here some properties of $f_Q$:
\begin{itemize}
	\item outside the convex hull of each box in $\mathcal F_Q$, the function $f_Q$ is equal to $f_Q^0$,
	\item $f_Q$ is $C\epsilon^{-1}$-Lipschitz for some universal constant $C$,
	\item $f_Q$ is non-negative.
\end{itemize}
\end{remark}
If $M_Q$ is not the standard $\epsilon$-box, we rescale  $f_Q(x)$ so that $C \epsilon^{-1}$-Lipschitzness is preserved.

\vv
Next, we prove a {result on the behavior of $f_Q$.}
\begin{lemma}
	\label{lemma:twopointswithdifferentfmu}
	{For} $z_1,z_2 \in \pom \cap I_Q$ with $\pi_{n+1}(z_1) < \pi_{n+1}({z_2}) - \mu^2 \diam M_Q$, we have {$f_Q(z_2)- f_Q(z_1) \gtrsim \pi_{n+1}(z_2) - \pi_{n+1}(z_1)$.} 
\end{lemma}
Note that the existence of $z_1$,$z_2\in \pom\cap I_Q$ with $\pi_{n+1}(z_1) < \pi_{n+1}({z_2}) - \mu \diam M_Q$ is a consequence of the definition of $\mathcal B(\epsilon,\mu)$.
\begin{proof}
	If any of the points $z_1,z_2$ is not contained inside the interior of a box {of the family} $\mathcal F_Q$, the result is an easy consequence of the fact that $f_Q = f_Q^0$ outside the convex hull of these {and that each box in $\mathcal F_Q$ has height much smaller than $\mu^2$}.
	Hence, let $M_1$ and $M_2$ be the largest boxes in $\mathcal F_Q$ such that their interior contains $z_1$ and $z_2$ respectively. Because of the choice of $\mathcal B_d(\epsilon)$ and the distance between $z_1$ and $z_2$, we have that $M_1 \neq M_2$ and {their tops must be well separated, i.e.} $\dist(T_1, T_2) \gtrsim \pi_{n+1}(z_2) - \pi_{n+1}(z_1)$. In turn, this implies that $f_Q(z_2) - f_Q(z_1) \gtrsim  \pi_{n+1}(z_2) - \pi_{n+1}(z_1)$. 
\end{proof}


The functions $f_Q$ for $Q\in\mathcal F$ have the following nice property that will be essential for what follows.
\begin{lemma}
For any choice of signs $(\epsilon_Q)_{Q\in\mathcal F}, \, \epsilon_Q \in \{\pm1\}$, we have that the function $f = \sum_{Q\in\mathcal F} \epsilon_Qf_Q$ is $C\epsilon^{-1}$-Lipschitz with constants independent of the particular choice of the family $\mathcal F$.
\end{lemma}
\begin{proof}
	Fix $Q\in \mathcal F$ and, without loss of generality, assume $M_Q$ is the standard $\epsilon$-box. The set where the function $f_Q$ has nonzero gradient can be decomposed as
	\[
	\{x \in \R^{n+1}\, : \, \nabla f_Q(x) \neq 0\} = G_{Q,1} \cup G_{Q,2}
	\] 
	where
	\begin{itemize}
		\item $G_{Q,1}$ is the interior of $M_Q$ and the inner base and sides intersected with $\{x_{n+1}\leq 9\epsilon\}$ minus the convex hull of the middle boundary of each box in $\mathcal F_Q$, that is
		\[
		G_{Q,1} = I_Q \cup  B_{Q,i}  \cup (S_{Q,i}\cap \{x_{n+1}\leq 9\epsilon\}) \backslash \cup_{Q_j\in\mathcal F_Q} \mbox{co}(S_{Q_j,m} \cup B_{Q_j,m}),
		\]
		\item $G_{Q,2}$ is the middle base and sides of each box in $\mathcal F_Q$ that is not completely contained inside the convex hull of a larger box in $\mathcal F_Q$, that is,
		\[
		G_{Q,2} = \bigcup_{Q_j\in\mathcal F_Q, \, f_Q^j \neq f_Q^{j-1}} (S_{Q_j,m}\cup B_{Q_j,m}).
		\]
	\end{itemize}
	It is easy to check that for $Q,R \in \mathcal F$, the set where both $\nabla f_Q$ and $\nabla f_R$ do not vanish is empty.  Hence, for any choice of signs $\epsilon_Q$, we obtain a uniformly Lipschitz function.
\end{proof}

\subsection{Properties of harmonic functions near boxes}
Consider a box $M_Q$, which without loss of generality we assume is the standard $\epsilon$-box. There are {three} possible behaviors of $\pom$ within the convex hull of $M_Q$:
\begin{enumerate}
	\item There exists $x\in \pom$ inside the convex hull of the inner sides that satisfies $\dist(x,T_Q) \leq \epsilon$ and the box $M_Q \subset \Omega$,
	\item all $x\in \pom$ inside the convex hull of the box $M_Q$ are inside $I_Q$ and the box $M_Q \subset \Omega$,
	\item the box $M_Q$ is entirely contained in $\Omega^c$.
\end{enumerate}
{We postpone the third possibility of $M_Q \subset \Omega^c$ until later.}
In the second case, define $P_Q := M_Q\, \cup\, \{x\in\mathbb R^{n+1}\,:\,  \sup_{i=1,\hdots, n} |x_i|\leq 1, \, \sup_{z\in I\cap\pom} \pi_{n+1}(z)<x_{n+1}< 10\epsilon \} \cup \{x\in\mathbb R^{n+1}\,:\,  \sup_{i=1,\hdots, n} |x_i|\leq 1, \, 0<x_{n+1}< \inf_{z\in I\cap\pom} \pi_{n+1}(z) \}$.
In the first case, let $\wt P$ be the projection onto the hyperplane $\{x_{n+1}=0\}$ of the set $\{x\in \mbox{convex hull of $M_Q$}\cap\pom\, : 
\, \dist(x,T_Q)\leq\epsilon\}$ and let $K$ be the convex hull of $\wt P$.  
Finally let $P_Q := M_Q \cup \{x\in\mathbb R^{n+1}\,:\, (x_1,\hdots,x_n) \not\in K,\, \sup_{i=1,\hdots,n}|x_i|\leq 1, \, 9\epsilon <x_{n+1}< 10\epsilon \}\cup \{x\in\mathbb R^{n+1}\,:\,  \sup_{i=1,\hdots, n} |x_i|\leq 1, \, 0<x_{n+1}< \inf_{z\in I\cap\pom} \pi_{n+1}(z) \}$.
In both cases $P_Q\subset \Omega$ is {an open connected set} (in the second case it resembles a closed box, and in the first case a closed box with a convex hole on the top) that satisfies a Poincar\'e inequality with constant possibly depending on $\epsilon$ but not on $\pom\cap I_Q$. Moreover, the bottom side of $P_Q$ is very close to the lowest point in $\pom\cap I$ and the top side of $P_Q$ is very close to $\pom$ as well. See Figure \ref{figure:example_of_PQ} for an example of the construction of $P_Q$ in the first case. 

\begin{figure}[h]

	\centering
	\begin{tikzpicture}[scale=5]
		\centering
		\def\eps{1/7}
		
		\pgfmathsetmacro{\epsthird}{\eps/3}
		\pgfmathsetmacro{\epstwoThird}{2*\eps/3}
		\pgfmathsetmacro{\oneMinusEpsThird}{1-\eps/3}
		\pgfmathsetmacro{\oneMinusTwoEpsThird}{1-2*\eps/3}
		\pgfmathsetmacro{\oneMinusEps}{1-\eps}
		\pgfmathsetmacro{\tenEps}{10*\eps}
		\pgfmathsetmacro{\twoEps}{2*\eps}
		
		
		
		\fill[gray!20] (-1,0) rectangle (1,\epsthird);
		
		\fill[gray!20] (-\oneMinusEpsThird,\epsthird) rectangle (\oneMinusEpsThird,\epstwoThird);
		
		\fill[gray!20] (-\oneMinusTwoEpsThird,\epstwoThird) rectangle (\oneMinusTwoEpsThird,\eps);
		
		\fill[gray!20] (-1,\epsthird) -- (-\oneMinusEpsThird,\epsthird) -- (-\oneMinusEpsThird,\tenEps) -- (-1,\tenEps) -- cycle;
		\fill[gray!20] (1,\epsthird) -- (\oneMinusEpsThird,\epsthird) -- (\oneMinusEpsThird,\tenEps) -- (1,\tenEps) -- cycle;
		
		\fill[gray!20] (-\oneMinusEpsThird,\epstwoThird) -- (-\oneMinusTwoEpsThird,\epstwoThird) -- (-\oneMinusTwoEpsThird,\tenEps) -- (-\oneMinusEpsThird,\tenEps) -- cycle;
		\fill[gray!20] (\oneMinusEpsThird,\epstwoThird) -- (\oneMinusTwoEpsThird,\epstwoThird) -- (\oneMinusTwoEpsThird,\tenEps) -- (\oneMinusEpsThird,\tenEps) -- cycle;
		
		\fill[gray!20] (-\oneMinusTwoEpsThird,\eps) -- (-\oneMinusEps,\eps) -- (-\oneMinusEps,\tenEps) -- (-\oneMinusTwoEpsThird,\tenEps) -- cycle;
		\fill[gray!20] (\oneMinusTwoEpsThird,\eps) -- (\oneMinusEps,\eps) -- (\oneMinusEps,\tenEps) -- (\oneMinusTwoEpsThird,\tenEps) -- cycle;
		
		
		
		\node[left] at (-0.5,0.3) {$\Omega$};
		\node[left] at (0.08,0.3) {$\Omega^c$};
		
		\fill[gray!20] (-1,10*\eps) rectangle (-9/14,9*\eps);
		\fill[gray!20] (1,10*\eps) rectangle (0.3,9*\eps);
		\draw[black, thick, domain=-0.4:0.5, smooth, variable=\x]
		plot ({\x}, {3*\x*\x + \eps});
		\draw[black, thick, domain=0.1:0.3, smooth, variable=\x]
		plot ({\x}, {4*\x+\eps});
		\draw[black, thick, domain=-0.7:-0.5, smooth, variable=\x]
		plot ({\x}, {-2*\x+\eps});
		\draw[black, thick, domain=0.1:0.5, smooth, variable=\x]
		plot ({\x}, {0.4+\eps+(\x-0.1)*0.88});
		\draw[black, thick, domain=-0.5:-0.4, smooth, variable=\x]
		plot ({\x}, {(1+\eps)-5.2*(\x+0.5)});
		\node[left] at (-0.85,1.34) {$P_Q$};
	\end{tikzpicture}
	\caption{{Example of the set $P_Q$ in the first case. Here $\pom$ is the black set and $P_Q$ the grey set.}}
		\label{figure:example_of_PQ}
\end{figure}
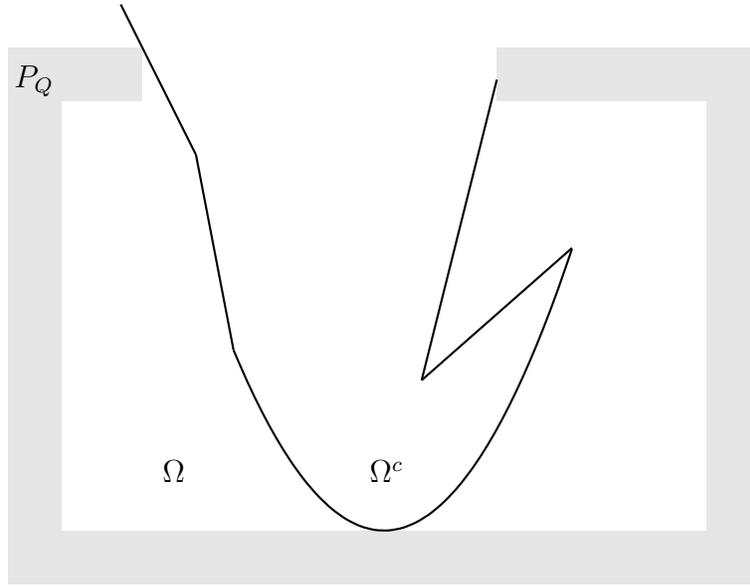

Let $u_Q$ be the solution of the Dirichlet problem on $\Omega$ with boundary data $f_Q$. 
\begin{lemma}
	\label{lemma:propertiesoffQ}
	We have
	\[
	\fint_{P_Q} |\nabla u_Q|^2 \, dm \gtrsim 1.
	\]
	with constants depending on $\epsilon$, $\mu$ {and the Ahlfors regularity constants of $\pom$ but not on $\pom$}.
\end{lemma}
\begin{proof}
Using Lemma \ref{lemma:Holderdecay} {(H\"older decay of solutions)}, the $n$-Ahlfors regularity of $\pom$ and that $f_Q$ is $C\epsilon^{-1}$-Lipschitz, we can find a small ball $b_{\text{top}} \subset P_Q$ with radius depending on $\epsilon$ and $\mu$ such that $u_Q|_{b_{\text{top}}}$ satisfies $|u_Q|_{b_{\text{top}}} - \inf_{z \in I\cap\pom} f_Q(z)| <\mu/10$.
We can do the same near the bottom of $P_Q$ and find a small ball $b_{\text{bot}} \subset P_Q$ where $|u_Q|_{b_{\text{bot}}} - \sup_{z \in I\cap\pom} f_Q(z)| <\mu/10$. Moreover, by {Lemma \ref{lemma:twopointswithdifferentfmu}} we know that $\sup_{z \in I\cap\pom}f_Q(z) - \inf_{z\in I \cap\pom}f_Q(z) \geq \mu$. Since $P_Q$ satisfies a Poincar\'e inequality,
we have
\[
\fint_{P_Q} |\nabla u_Q|^2 \, dm \gtrsim_\epsilon \ell(Q)^{-2}\fint_{P_Q} |u - u_{P_Q}|^2\, dm \gtrsim_{\mu}\ell(Q)^{-2} \fint_{b_{\text{top}}\cup b_{\text{bot}}} |u - u_{P_Q}|^2\, dm \gtrsim_\mu 1.
\]
\end{proof}
\vv

{Finally, we deal with the case $M_Q\subset \Omega^c$.} Without loss of generality, assume $\mu$ is much smaller than the corkscrew constant $C_{\operatorname{ck}}$ of the domain $\Omega$.
By the definition of $\mathcal B(\epsilon, \mu)$ there exist $z_1,z_2 \in \pom \cap I_Q$ with $\pi_{n+1}(z_1) < \pi_{n+1}(z_2) - \mu$. This implies that the corkscrew ball $B_Q \subset B(z_2,\mu) \cap \Omega$ (which has radius $\mu \cdot C_{\operatorname{ck}}$) is contained inside $I_Q$.
Let $x_Q$ be the center of $B_Q$, $\mathcal C_1$ denote the cylinder that starts at $x_Q$, with direction $-e_1$ (the first axis), radius $\mu\cdot C_{\operatorname {ck}}/4$ and that extends until it first intersects $\pom$.
Define $\mathcal C_{n+1}$ analogously but in the direction $-e_{n+1}$ (the $n+1$-th axis).
Note that these cylinders contain $x_Q \in \Omega$ and are directed towards $M_Q \subset  \Omega^c$, hence they must meet $\pom\cap I_Q$. 

Let $\xi_1, \xi_{n+1} \in \pom \cap I_Q$ be the touching points of $\mathcal C_1$, $\mathcal C_{n+1}$ respectively. Note that $\pi_{n+1}(\xi_1) < \pi_{n+1}(\xi_{n+1}) - \mu C_{\operatorname{ck}}/50$. Hence, by Lemma \ref{lemma:twopointswithdifferentfmu}, we have that $f(\xi_{n+1}) - f(\xi_1) \gtrsim C_{\operatorname{ck}}\mu$.
Note also that the open set $P_Q := B_Q \cup \mathcal C_1 \cup \mathcal C_n$ is contained in $\Omega$ and satisfies a Poincar\'e inequality with constant depending on $\mu C_{\operatorname{ck}}$.  Hence, we can easily adapt the proof of Lemma \ref{lemma:propertiesoffQ} to show that it also holds in this setting. See Figure \ref{figure:example_of_PQ_case3} for an example of the construction of $P_Q$ in this case.
\color{black}

\begin{figure}[h]
	
	\centering
	\begin{tikzpicture}[scale=5]
		\centering
		\def\eps{1/7}
		
		\pgfmathsetmacro{\epsthird}{\eps/3}
		\pgfmathsetmacro{\epstwoThird}{2*\eps/3}
		\pgfmathsetmacro{\oneMinusEpsThird}{1-\eps/3}
		\pgfmathsetmacro{\oneMinusTwoEpsThird}{1-2*\eps/3}
		\pgfmathsetmacro{\oneMinusEps}{1-\eps}
		\pgfmathsetmacro{\tenEps}{10*\eps}
		\pgfmathsetmacro{\twoEps}{2*\eps}
		
		
		
		\fill[gray!20] (-1,0) rectangle (1,\epsthird);
		
		\fill[gray!20] (-\oneMinusEpsThird,\epsthird) rectangle (\oneMinusEpsThird,\epstwoThird);
		
		\fill[gray!20] (-\oneMinusTwoEpsThird,\epstwoThird) rectangle (\oneMinusTwoEpsThird,\eps);
		
		\fill[gray!20] (-1,\epsthird) -- (-\oneMinusEpsThird,\epsthird) -- (-\oneMinusEpsThird,\tenEps) -- (-1,\tenEps) -- cycle;
		\fill[gray!20] (1,\epsthird) -- (\oneMinusEpsThird,\epsthird) -- (\oneMinusEpsThird,\tenEps) -- (1,\tenEps) -- cycle;
		
		\fill[gray!20] (-\oneMinusEpsThird,\epstwoThird) -- (-\oneMinusTwoEpsThird,\epstwoThird) -- (-\oneMinusTwoEpsThird,\tenEps) -- (-\oneMinusEpsThird,\tenEps) -- cycle;
		\fill[gray!20] (\oneMinusEpsThird,\epstwoThird) -- (\oneMinusTwoEpsThird,\epstwoThird) -- (\oneMinusTwoEpsThird,\tenEps) -- (\oneMinusEpsThird,\tenEps) -- cycle;
		
		\fill[gray!20] (-\oneMinusTwoEpsThird,\eps) -- (-\oneMinusEps,\eps) -- (-\oneMinusEps,\tenEps) -- (-\oneMinusTwoEpsThird,\tenEps) -- cycle;
		\fill[gray!20] (\oneMinusTwoEpsThird,\eps) -- (\oneMinusEps,\eps) -- (\oneMinusEps,\tenEps) -- (\oneMinusTwoEpsThird,\tenEps) -- cycle;
		
		
		
		\node[left] at (-0.64,0.3) {$\Omega^c$};
		\node[left] at (0,1.2) {$\Omega$};
		
		\draw (-0.1,0.7) circle (0.15);
		\node[left] at (-0.04,0.78) {$B_Q$};
		
		\draw (-0.1,0.7+0.15/4) -- (-0.405,0.7+0.15/4);
		
		\node[left] at (-0.26,0.78) {$\mathcal C_1$};
		
		\filldraw (-0.405,0.7-0.15/4) circle (0.25pt);
		\node[left] at (-0.405-0.02,0.7-0.15/4) {$\xi_1$};
		
		\draw (-0.1,0.7-0.15/4) -- (-0.405,0.7-0.15/4);
		
		\draw (-0.405,0.7-0.15/4) -- (-0.405,0.7+0.15/4);
		
		\draw (-0.1,0.7-0.15/4) -- (-0.1,0.7+0.15/4);
		
		\draw (-0.1-0.15/4,0.7) -- (-0.1-0.15/4,0.2);
		
		\node[left] at (0.15,0.46) {$\mathcal C_{n+1}$};
		
		\filldraw (-0.1-0.15/4,0.2) circle (0.25pt);
		\node[left] at (-0.1-0.15/4-0.02,0.2) {$\xi_{n+1}$};
		
		\draw (-0.1+0.15/4,0.7) -- (-0.1+0.15/4,0.2);
		
		\draw (-0.1+0.15/4,0.2) -- (-0.1-0.15/4,0.2);
		
		\draw (-0.1-0.15/4,0.7) -- (-0.1+0.15/4,0.7);
		
		\draw[black, thick, domain=-0.4:0.5, smooth, variable=\x]
		plot ({\x}, {3*\x*\x + \eps});
		\draw[black, thick, domain=0.1:0.3, smooth, variable=\x]
		plot ({\x}, {4*\x+\eps});
		\draw[black, thick, domain=-0.7:-0.5, smooth, variable=\x]
		plot ({\x}, {-2*\x+\eps});
		\draw[black, thick, domain=0.1:0.5, smooth, variable=\x]
		plot ({\x}, {0.4+\eps+(\x-0.1)*0.88});
		\draw[black, thick, domain=-0.5:-0.4, smooth, variable=\x]
		plot ({\x}, {(1+\eps)-5.2*(\x+0.5)});
	\end{tikzpicture}
	\caption{{Example of the set $P_Q$ formed by the ball $B_Q$ and the two cylinders $\mathcal C_1$ and $\mathcal C_{n+1}$ in the third case.}}
	\label{figure:example_of_PQ_case3}
\end{figure}
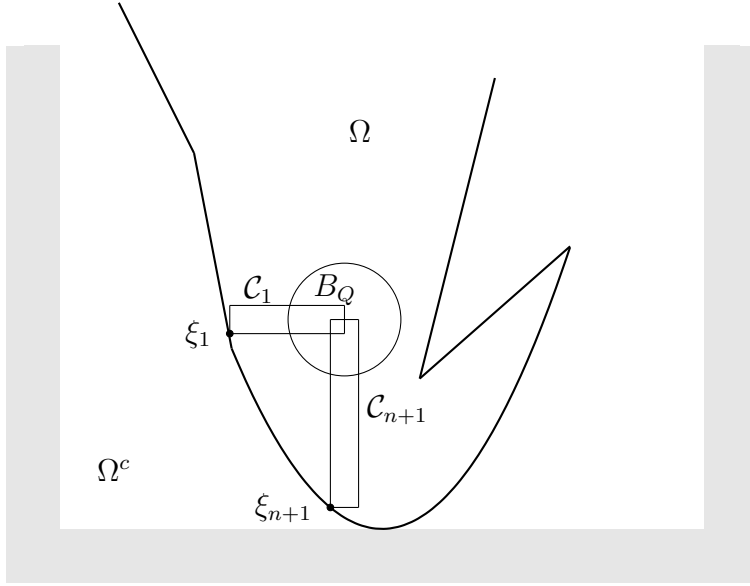
\subsection{Counterexample to one-sided Rellich inequality}
For some $\epsilon>0$ we assume $\mathcal B_d(\epsilon)$ is not Carleson. By the results in \cite[Section 16]{NTV}, for every integer $k>0$, we can find $Q_0 \in \mathcal B_d(\epsilon)$ (depending on $k$) and $k$ families of subcubes $(\mathcal F_i)_{i=1}^k$, $\mathcal F_i \subset \mathcal B_d(\epsilon)$ satisfying 
\begin{enumerate}
	\item for every cube $Q\in \mathcal F_i$, there exists a cube $R \in \mathcal F_{i-1}$ such that $Q\subsetneq R$, for $i=2,\hdots,k$,

	\item every cube in $\mathcal F_1$ is contained inside $Q_0$,
	\item the cubes in $\mathcal F_i$ are pairwise disjoint, for $i=1,\hdots, k$,
	\item $\sigma(\cup_{Q\in \mathcal F_k} Q) \geq \frac 1 2 \sigma(Q_0)$.
\end{enumerate}
Let $\mathcal F = \bigcup_{i=1}^k \mathcal F_i \cup \{Q_0\}$ and consider the construction  of $f_Q$'s for $Q\in \mathcal F$ {from Section \ref{section:constructionbadfunctions}}. {For the sake of contradiction, we also assume that the one-sided Rellich inequality \eqref{eq:rellich_ineq} holds for all $f$ Lipschitz. In particular, we suppose $\partial_\nu u_{f_Q}$ is a Radon measure.}
\begin{lemma}
	\label{lemma:randomfunctionsforeachlayer}
	Fix $j\in \{1,\hdots,k\}$. For a random choice of signs $(\epsilon_Q)_{Q\in \mathcal F_j}$, let ${f_j} = \sum_{Q\in\mathcal F_j} \epsilon_Q f_Q$. We have
	\[
	\mathbb E \left[\sum_{Q\in\mathcal F_j} |\partial_\nu u_{f_j}|(Q)\right] \gtrsim \sigma(Q_0)
	\]
	where $u_{f_j}$ is the solution to the Dirichlet problem with boundary data $f_j$ and the inequality has constants independent of $j$ and $k$.
	In particular, for any $j$, there exists a particular choice of signs $(\epsilon_{Q}^j)_Q$ such that $f_j^* = \sum_{Q\in\mathcal F_j} \epsilon_Q^j f_Q$ satisfies 
	\[
	\sum_{Q\in\mathcal F_j} |\partial_\nu u_{f_j^*}|(Q) \gtrsim \sigma(Q_0).
	\]
\end{lemma}
\begin{proof}
	Let $\wt f_j = \sum_{Q\in\mathcal F_j} \epsilon_Q f_Q/\ell(Q)$.
	For any choice of signs, using duality and that $\wt f_j$ has constant sign in every $Q\in\mathcal F_j$, we have
	\[
	\int_\Omega \nabla u_{f_j} \nabla u_{\wt f_j}\, dm = \int_{\pom} \partial_\nu u_{f_j} {\wt f_j}\, d\sigma \leq
	\sum_{Q\in\mathcal F_j} |\partial_\nu u_{f_j}|(Q) \Vert {\wt f_j} \Vert_{L^\infty(Q)}  \approx_\mu
	\sum_{Q\in\mathcal F_j} |\partial_\nu u_{f_j}|(Q).
	\]
	Using {Khintchine's} inequality \ref{prop:Khintchine}, and {Lemma \ref{lemma:propertiesoffQ}}, we have
	\begin{align*}
		\mathbb E \int_{\Omega} \nabla u_{f_j} \nabla u_{\wt f_j} \,dm 
		&\geq \sum_{Q\in\mathcal F_j} \mathbb E\int_{ P_Q} \nabla u_{f_j} \nabla u_{\wt f_j}\, dm 
		\approx \sum_{Q\in\mathcal F_j} \int_{P_Q} \left ( \sum_{R \in\mathcal F_j} |\nabla u_{R}|^2\ell(R)^{-1} \right)\, dm \\
		&\geq \sum_{Q\in\mathcal F_j} \int_{ P_Q}|\nabla u_{Q}|^2 \ell(Q)^{-1}\, dm 
		\gtrsim \sum_{Q\in\mathcal F_j} \sigma(Q) \geq \frac 1 2  \sigma(Q_0)
		.
	\end{align*}
	Hence, we obtain
	\[
	\mathbb E\left[\sum_{Q\in\mathcal F_j} |\partial_\nu u_{f_j}|(Q)\right]\gtrsim \sigma(Q_0)
	\]
	and we can find a particular choice of signs $(\epsilon_{Q}^j)_{Q\in\mathcal F_j}$ such that $f_j^* = \sum_{Q\in\mathcal F_j} \epsilon_Q^j f_Q$ satisfies  
	\[
	\sum_{Q\in\mathcal F_j} |\partial_\nu u_{f_j^*}|(Q) \gtrsim \sigma(Q_0).
	\]
\end{proof}

\begin{remark}
	Note that $\sum_{Q\in\mathcal F_j} |\partial_\nu u_{f_j^*}|(Q) \leq \Vert  \partial_\nu u_{f_j^*} \Vert_{\mathcal M}.$
\end{remark}

%
\begin{proof}[Proof of Theorem \ref{thm:RellichimpliesWNB}]
	Assume $\pom$ does not satisfy WNB. Let $k>0$ and consider $Q_0^k$ and $\mathcal F_1, \hdots, \mathcal F_k$ as above. We will show that a function of the form $f = \sum_{j=1}^k \epsilon_j f_j^*$ for a random choice of signs $(\epsilon_j)_{j=1}^k$ satisfies in expectation $\mathbb E [\Vert \partial_\nu u_f \Vert_{\mathcal M}]\gtrsim \sqrt k \sigma(Q_0^k) \gtrsim \sqrt k \Lip(f) \sigma(\supp f) \gtrsim \sqrt k \Vert \nabla_H f \Vert_{L^1}$ with constants independent of $k$.
	
	Let $\mathcal F' = \{Q\in\Dsigma\,:\, Q\in \mathcal F_k \mbox{ or }Q \mbox{ maximal cube in } Q_0^k\backslash \mathcal F_k\}$ and
	fix $Q\in\mathcal F'$. By Khintchine's inequality \ref{prop:Khintchine} and Cauchy-Schwarz, we have 
	\[
	\mathbb  E \left|\sum_{j}^k \epsilon_j \partial_\nu u_{f_j^*} (Q)\right|
	\approx \sqrt{\sum_j^k |\partial_\nu u_{f_j^*}|(Q)^2} \geq
	\frac 1 {\sqrt k} \sum_j^k |\partial_\nu u_{f_j^*}|(Q).
	\]
	Using {Lemma \ref{lemma:randomfunctionsforeachlayer}}, we obtain
	\[
	\mathbb  E\sum_{Q\in\mathcal F'} \left| \sum_{j}^k \epsilon_j \partial_\nu u_{f_j^*} (Q)\right|\, 
	\gtrsim \frac{1}{\sqrt k}		
	\sum_{j}^k\sum_{Q \in\mathcal F'} |\partial_\nu u_{f_j^*}(Q)| \,  \gtrsim \sqrt{k}\sigma(Q_0^k).
	\]
	Hence, there is at least one combination of signs $(\epsilon_j)_{j=1}^k$ for which $\Vert \partial_\nu u_f \Vert_{\mathcal M} \gtrsim \sqrt{k} \sigma(Q_0^k)$. 
\end{proof}

\appendix

\section{One-sided Rellich inequalities in $L^p$ imply weak-$\mathcal A_\infty$ property for harmonic measure}
\label{appendix:RellichLpimpliesReverseHolder}
{In this appendix, we show the relationship between the one-sided Rellich inequalities in $L^q$ and the weak-$\mathcal A_\infty$ properties of the harmonic measure of the domain. The converse direction can be found in \cite{MT}.}
%
%
%
\begin{proposition}
	Let $\Omega\subset \R^{n+1}$, $n\geq2$ be a bounded open set with $n$-Ahlfors regular boundary. Assume that for all $f\in\Lip(\pom)$ and $u_f$ solution to the Dirichlet problem for the Laplacian, the one-sided Rellich inequality
	\[
	\Vert \partial_\nu u_f\Vert_{L^q(\pom)} \lesssim  \Vert \nabla_H f \Vert_{L^q(\pom)}
	\]
	holds for some $q>1$. Then the Dirichlet problem $(D^\Delta_{q'})$ is solvable where $\frac 1 {q'} + \frac 1 q = 1$. 
\end{proposition}

\begin{proof}
	For $p\in\Omega$, let $f_p(\xi) := \mathcal E(p-\xi)$ for $\xi\in\pom$. It is easy to show that
	$\Vert \nabla_H f\Vert_{L^q}^q \lesssim \dist(p,\pom)^{-n(q-1)}$ (see for example the proof of \cite[Theorem 1.6]{MT}). On the other hand, for $\wt g \in \Lip(\mathbb R^{n+1})$ with compact support and 
	using \cite[Lemma 7.6]{PT}, we have
	\[
	\int_\pom \partial_\nu u_f \wt g\, d\sigma =  \int_{\Omega} \nabla \wt g \cdot \nabla u_f\, dx = \int_\Omega \nabla \wt g \left ( \nabla_x \mathcal E(p-x) -\nabla_x G(p,x) \right)\, dx
	\]
	\[
	=-\int_{\pom_*}	\wt g \partial_\nu \mathcal E (p-x) \, d\HH^n|_{\pom_*} + \int_\pom \wt g d\omega^p
	\]
	where $\pom_*$ is the measure theoretic boundary of $\Omega$. 
	Since $\partial_\nu \mathcal E(p-x)\leq |\nabla\mathcal E(p-x)|$ $\sigma$-a.e. on $\pom_*$,
	by duality we obtain that the density $\frac{d\omega^p}{d\sigma}$ exists, is in $L^q(\sigma)$, and 
	\[
	\left\Vert \frac{d\omega^p}{d\sigma} \right\Vert_{L^q(\pom)} \lesssim \Vert \nabla \mathcal E(p-\cdot) \Vert_{L^q(\pom_*)} + \Vert \nabla_H \mathcal E(p-\cdot) \Vert_{L^q(\pom)} \leq \dist(p,\pom)^{-n(q-1)/q}.
	\] 
	In particular, the density of the harmonic measure satisfies  \cite[Theorem 9.2 (c)]{MT}, which is equivalent to $(D_{q'}^\Delta)$ being solvable. 
\end{proof}

\end{document}